\documentclass[11pt]{amsart}
\usepackage[usenames,dvipsnames]{color}
\usepackage{latexsym,xspace,enumerate,amsfonts,amsmath,amssymb,amscd}
\usepackage{xypic,tikz,hyperref,syntonly,MnSymbol,slashed,paralist}
\usepackage{marginnote}
\reversemarginpar
\usepackage[capitalize]{cleveref}
\usepackage[arrow, matrix, curve]{xy}

\newcommand{\Z}{\mathbb{Z}}

\newcommand{\R}{\mathbb{R}}
\newcommand{\C}{\mathbb{C}}

\newcommand{\SL}{\mr{SL}}

\newcommand{\su}{\mathfrak{su}}

\newcommand{\mf}{\mathfrak}
\newcommand{\mr}{\mathrm}

\newcommand{\mc}{\mathcal}

\newcommand{\End}{\mathop{\rm End}\nolimits}

\renewcommand{\ker}{\mathop{\rm ker}\nolimits}
\newcommand{\im}{\mathop{\rm im}\nolimits}

\renewcommand{\Re}{\mathop{\rm Re}\nolimits}
\renewcommand{\Im}{\mathop{\rm Im}\nolimits}
\renewcommand{\deg}{\mathop{\rm deg}\nolimits}

\newcommand{\Tr}{\mathop{\rm Tr}\nolimits}

\newcommand{\del}{\partial}
\newcommand{\delb}{\bar{\partial}}

\newcommand{\note}[1]{\marginpar{\raggedright\if@twoside\ifodd\c@page\raggedleft\fi\fi\sf\scriptsize \red{RMK: #1}}}

\newcommand\red[1]{\textcolor{red}{#1}}

\newcommand{\be}{\begin{equation}}
\newcommand{\ben}{\begin{equation}\nonumber}
\newcommand{\ee}{\end{equation}}
\newcommand{\bp}{\begin{para}}
\newcommand{\ep}{\end{para}}
\newcommand{\RR}{\mathbb R}
\newcommand{\CC}{\mathbb C}

\newcommand{\fid}{\mathrm{fid}}

\newcommand{\supp}{\mathrm{supp}\,}
\newcommand{\calC}{{\mathcal C}}

\newcommand{\calF}{{\mathcal F}}

\newcommand{\calK}{{\mathcal K}}
\newcommand{\calL}{{\mathcal L}}
\newcommand{\calM}{{\mathcal M}}

\newcommand{\calB}{{\mathcal B}}
\newcommand{\calO}{{\mathcal O}}
\newcommand{\calQ}{{\mathcal Q}}

\newcommand{\calS}{{\mathcal S}}
\newcommand{\calT}{{\mathcal T}}

\newcommand{\calW}{{\mathcal W}}

\newcommand{\appr}{\mathrm{app}}
\newcommand{\exterior}{\mathrm{ext}}

\newcommand{\Prym}{\mathrm{Prym}}

\newcommand{\sfl}{\mr{sf}}
\newcommand{\app}{\mathrm{app}}
\newcommand{\SW}{\mathrm{SW}}
\newcommand{\sK}{\mathrm{sK}}

\def\im{\textrm{im}\,}

\def\D{\mathbb{D}}
\def\odd{\mathrm{odd}}
\def\hor{\mathrm{hor}}
\def\verti{\mathrm{vert}}

\def\skew{\mathrm{skew}}
\def\GMN{\mathrm{GMN}}

\newsavebox{\dotbox}

\newtheorem{proposition}{\textbf{Proposition}}
\newtheorem{lemma}[proposition]{\textbf{Lemma}}
\newtheorem{corollary}[proposition]{\textbf{Corollary}}
\newtheorem{theorem}[proposition]{\textbf{Theorem}}

\theoremstyle{definition}
\newtheorem{definition}{\textbf{Definition}}

\newtheorem*{example*}{\textbf{Example}}
\newtheorem*{remark}{\textbf{Remark}}
\theoremstyle{remark}      

\numberwithin{proposition}{section}
\numberwithin{definition}{section}
\begin{document}
\title[Asymptotic geometry of the Hitchin metric]{Asymptotic geometry of the Hitchin metric}
\date{\today}

\author{Rafe Mazzeo}
\address{Department of Mathematics, Stanford University, Stanford, CA 94305 USA}
\email{mazzeo@math.stanford.edu}

\author{Jan Swoboda}
\address{Mathematisches Institut der Universit\"at M\"unchen\\Theresienstra{\ss}e 39\\D--80333 M\"unchen\\ Germany}
\email{swoboda@math.lmu.de}

\author{Hartmut Wei\ss}
\address{Mathematisches Seminar der Universit\"at Kiel\\ Ludewig-Meyn-Stra{\ss}e 4\\ D--24098 Kiel\\ Germany}
\email{weiss@math.uni-kiel.de}

\author{Frederik Witt} 
\address{Institut f\"ur Geometrie und Topologie der Universit\"at Stuttgart\\ Pfaffenwaldring 57\\ D--70569 Stuttgart\\ Germany}
\email{frederik.witt@mathematik.uni-stuttgart.de}

\thanks{RM supported by NSF Grant DMS-1105050 and DMS-1608223.}
\thanks{JS \& HW supported by DFG SPP 2026 `Geometry at infinity'.}  
\thanks{The author(s) acknowledge(s) support from U.S. National Science Foundation grants DMS 1107452, 1107263, 1107367 "RNMS: Geometric Structures and Representation Varieties" (the GEAR Network).}
\maketitle

\begin{abstract}
We study the asymptotics of the natural $L^2$ metric on the Hitchin moduli space with group $G = \mathrm{SU}(2)$. Our 
main result, which addresses a detailed conjectural picture made by Gaiotto, Neitzke and Moore \cite{gmn13}, is that on 
the regular part of the Hitchin system, this metric is well-approximated by the semiflat metric from \cite{gmn13}.  We prove
that the asymptotic rate of convergence for gauged tangent vectors to the moduli space has a precise polynomial expansion,
and hence that the   difference between the two sets of metric coefficients in a certain natural coordinate system also has 
polynomial decay.  New work by Dumas and Neitzke shows that the convergence is actually exponential in directions  tangent to the Hitchin section. 
\end{abstract} 
%
%
%%%%%%%%%%%%%%%%%%%%%%%%%%%%%%%%%%%%%%%%%%%%%%%%%
%%%%%%%%%%%%%%%%%%%%%%%%%%%%%%%%%%%%%%%%%%%%%%%%%
\section{Introduction}
%%%%%%%%%%%%%%%%%%%%%%%%%%%%%%%%%%%%%%%%%%%%%%%%%
%%%%%%%%%%%%%%%%%%%%%%%%%%%%%%%%%%%%%%%%%%%%%%%%%
In this paper we study the asymptotic geometry of the $L^2$ (`Weil-Petersson type') metric $g_{L^2}$ on the moduli space 
$\calM_{2,d}$ of irreducible solutions to the Hitchin self-duality equations on a $\mathrm{SU}(2)$-bundle $E$ of degree $d$
over a compact Riemann surface $X$, modulo unitary gauge transformations.  We often refer to $g_{L^2}$ as the Hitchin
metric on $\calM_{2,d}$. The space $\calM_{2,d}$ can also be identified as the moduli space of  stable Higgs bundles $(A, \Phi)$
modulo complex gauge transformations, as well as the twisted character variety of irreducible representations of $\pi_1(X)$
into $\mathrm{SL}(2,\CC)$ modulo conjugation.  The fact that $g_{L^2}$ is hyperk\"ahler reflects these various realizations.
All of this can be generalized to the situation where $E$ has higher rank and carries a $G$-structure, where $G$ is
any compact semisimple group.  We treat here only the case $G = \mbox{SU($2$)}$ and for simplicity also set $d = 0$
(the differences needed to handle $d \neq 0$ are minor).  We denote the moduli space simply $\calM$.

Many topological and geometric properties of $\calM$ are now understood, and in the past few years a detailed picture
has started to emerge about its asymptotic geometric structure at infinity.   A key role 
in this story is played by the space $\calM_\infty$ of `limiting configurations', consisting of the solutions of a set
of decoupled equations which arise as a limiting form of the Hitchin equations, again modulo unitary gauge transformations.
 There are proper surjective mappings $\pi: \calM \to \calB$ and $\pi_\infty: \calM_\infty   \to \calB$ onto
  the space of holomorphic quadratic differentials; each carries a Higgs bundle $(A,\Phi)$ to $\det \Phi$. 
  The subset $\calM_\infty'$ of limiting configurations over the `free region' $\mathcal B' \subset \mathcal B$ of quadratic
  differentials with only simple zeroes was introduced in \cite{msww14}; later, Mochizuki \cite{moch15} extended the definition
  of limiting configurations to include those also lying over the `discriminant locus' $\Lambda = \calB \setminus \calB'$.
  We denote the preimage $\pi^{-1}(\calB')$ by
  $\calM'$. There is a canonical diffeomorphism 
\begin{equation}
\calF: \calM'_\infty \longrightarrow \calM'
  \label{MinftytoM}
  \end{equation}
  which we explain later.  This diffeomorphism allows us to transfer functions, vector fields and tensors from $\calM_\infty'$
  to $\calM'$ and back. The maps $\pi$ and $\pi_\infty$ are quadratic in the Higgs field, so the natural $\CC^\times$ action
  on Higgs bundles $(A, \Phi)$ satisfies $\pi( A, t\Phi) = t^2 \det \Phi$, and similarly for $\pi_\infty$.  We consider here only
  the restriction of this $\CC^\times$ action to an $\RR^+$ action.  The space $\calB'$ is a cone with respect to this (quadratic)
  action, while $\calM_\infty'$ is `semi-conic', i.e., it is a bundle of tori over the cone $\calB'$ where the fibers along each
  $\RR^+$ orbit are all the same.    Limiting configurations are one of the two building blocks for the construction of diverging 
families of solutions in $\calM'$ \cite{msww14}  (the other is the family of fiducial solutions, cf.\ \S 4.)

 Entirely distinct from those developments, motivated by supersymmetric quantum field theory, a beautiful
  conjectural picture of the asymptotic geometry of $\calM$ has  been established in the monumental work by Gaiotto,
  Moore and Neitzke \cite{gmn13}. These authors develop the formalism of spectral networks on Riemann surfaces, out of which
they construct a hyperk\"ahler metric $g_{\GMN}$ on $\calM$ which they conjecture to be precisely equal to the metric $g_{L^2}$.
The short survey paper by Neitzke \cite{ne14} contains an overview of this construction.    Part of their
  story involves a simpler hyperk\"ahler metric $g_{\sfl}$ on $\calM'$, called the semiflat metric, which is
  canonically associated to the underlying algebraic completely integrable system structure. They show that it is a good
  approximation to $g_{\GMN}$ in the sense that 
\[
g_{\GMN} \sim g_{\sfl} + \calO( e^{-\beta t}).
\]
The error term is a symmetric two-tensor whose norm with respect to $g_{\sfl}$ decays at the stated rate, where we
  are identifying the dilation parameter $t$ as a radial variable on $\calM'$, and the exponential decay rate $\beta$
  depends on the particular $\RR^+$ orbit and degenerates as this ray converges to $\calB \setminus \calB'$.
(There is a precise conjectured formula for $\beta$ which we do not state here.)  

These two points of view lead to the challenge of understanding the  Gaiotto-Moore-Neitzke metric and its relationship to
$g_{L^2}$.  This is the goal of the present paper.  In more detail, we have two main results.  
\begin{theorem}
  The pullback $\calF^* g_{\sfl}$ of the semiflat metric to $\calM_\infty'$ is
    a renormalized $L^2$ metric on $\calM_\infty'$. 
\label{semiflattheorem}
\end{theorem}

 The diffeomorphism $\calF$ can be defined via the Kobayashi-Hitchin correspondence,
  since points on $\calM'$ and $\calM_\infty'$ are each associated to unique points of the complex
  gauge group orbit (modulo the real gauge group).  Alternately, at least outside of a large ball, it can also
  be defined via the construction in \cite{msww14} of `large' solutions to the Hitchin equations. Furthermore,
  there are natural maps from $T^* \calB'/\Gamma$ to both $\calM'$ and $\calM_\infty'$. Here $\Gamma$ is
  a certain local system of lattices over $\calB'$ which can be described either cohomologically or
  using the algebraic completely integrable system structure on $\calM$.  Thus all three spaces are
  naturally identified and it is more or less a matter of taste which one of these one considers the most fundamental.
  Both $T^* \calB'/\Gamma$ and $\calM_\infty'$ have more obvious coordinates, and these induce
  coordinates on $\calM'$. It is in terms of these that we write the metric coefficients for $g_{L^2}$ and $g_{\sfl}$
  later. Our second result quantifies the sense in which these are close: 
 \begin{theorem}
There is a convergent series expansion 
\[
 g_{L^2} = g_{\sfl} + \sum_{j=0}^\infty  t^{(4-j)/3} G_j + \calO(e^{-\beta t}) 
\]
as $t \to \infty$, where each $G_j$ is a dilation-invariant symmetric two-tensor.  The rate $\beta> 0$ of exponential decrease
of the remainder is uniform in any closed dilation-invariant sector $\overline{\mathcal W} \subset \calM'_\infty$ disjoint
from $\pi_{\infty}^{-1}(\calB \setminus\calB')$. 
\label{expansiontheorem}
\end{theorem}
 The terms in this series are all lower order, including those 
  with positive powers of $t$. Indeed, the semi-conic nature of $g_{\sfl}$ means that its horizontal metric coefficients
  (relative to $\pi_\infty$) grow like $t^2$, and the $G_j$ with $j \leq 4$ are only nonvanishing in those directions. 
  
 Throughout this article we say that a tensor $G$ on $\calM_\infty'$ is {\it polynomial in $t$} if it
  has the form $G=t^{\alpha}G'$ for some   real number $\alpha$, where $G'$ is dilation invariant, or slightly
  more generally, if it has a convergent expansion in terms of such monomial terms. 

\begin{remark}
  The polynomial correction terms in Theorem~\ref{expansiontheorem} arise in a natural way. The calculations which
    produce gauged tangent vectors to the moduli space and the corresponding metric coefficients lead to expressions of the form 
\[
\int_{\mathbb D} f(t^{2/3} z) \frac{\dot{q}}{q}
\] 
where $q$ and $\dot{q}$ are holomorphic quadratic differentials, $z$ is a local holomorphic coordinate in a disk
$\mathbb D$ 
centered at a zero of $q$, and $f$ is a $\calC^\infty$ function which decays exponentially in its argument, 
or more generally, a convergent sum of such functions. The quotient $\dot{q}/q$ is meromorphic in $z$ with
a simple pole at $z=0$ (provided $q$ has simple zeroes and $\dot{q}$ does not vanish at these zeroes).  A simple calculation shows that these integrals
lead to asymptotic expressions in $t$ as above.  The precise calculations appear in Sections 5 and later.  
\end{remark}

In light of the prediction that $g_{L^2} - g_{\sfl}$ decays exponentially in $t$, it is of considerable interest to
determine whether any of these polynomial correction terms $ G_j$ are nonzero.
% \textcolor{blue}{Proving the vanishing  of all coefficients would support this conjecture but would only show that the decay is rapid, not necessarily exponential.}
Although the basic strategy and many of the technical aspects of this paper were understood by us two or three 
years ago, it was written slowly and its final release was delayed for some months as 
we investigated the sharpness of our results. Around the time this paper was posted, 
David Dumas and Andy Neitzke announced some further progress, which has just now appeared \cite{dn18}. In this, they 
explain a remarkable cancellation that takes place in the   difference of metric coefficients in `horizontal' 
directions tangent to  the Hitchin section. 
This is then transfered to show the exponential convergence of the horizontal components of $g_{L^2}$ to $g_{\sf}$ 
on the Hitchin section over a general compact Riemann 
surface $X$.  This is accomplished with careful attention to the rate of exponential decay, but unfortunately they miss the 
conjectured  sharp numerical value of this rate by a factor of $2$.  Their result has successfully been extended to the  entire space $\mathcal M'$, including   non-horizontal directions and the region  off of the Hitchin section, in the very recent preprint \cite{fr18} by Laura Fredrickson. 

%In any case, the meaning of the computations in \cite{dn18}   leading to this cancellation is not yet clear and presents
%a very interesting challenge for future work. 

The techniques of the present paper lead to a number of other interesting results, and we hope the approach
developed here will be useful in a number of related problems.

  We note in particular that even though the relative 
decay rate of the metric asymptotics has now been proven  to be exponential everywhere on $\mathcal M'$,  one sees, using Proposition \ref{prop:normalphatphi} below, that gauged tangent vectors themselves converge to their limits only at a polynomial rate. 

The terminology and basic definitions needed to fill out the brief discussion above will be presented in the next two sections.
Following that, we study the deformations of the space of limiting configurations and prove Theorem~\ref{semiflattheorem}.
On the actual moduli space, one of the main technical issues is to put infinitesimal deformations of a given solution
into gauge.  The special types of fields encountered here which arise in this gauge-fixing require some novel mapping
properties of the inverse of the `gauge-fixing operator' $\calL_t$. These are proved in \S 5. The remaining sections
use this to systematically compute the metric coefficients in various directions, which establishes Theorem~\ref{expansiontheorem}.

The authors wish to extend their thanks to a number of people with whom we had very helpful conversations. The two who
should be singled out are Nigel Hitchin and Andy Neitzke, both of whom contributed substantially, both in terms
of encouragement and their very thoughtful advice at various stages. We also thank Laura Fredrickson and Sergei Gukov for many 
insightful remarks and Steven Rayan for a very thorough reading of a first draft of the paper.  
 Finally, we are also extremely grateful to the referee for an extraordinarily detailed report which led to many
clarifications of the text, and also for pointing out the reference \cite{dh75}. 

%%%%%%%%%%%%%%%%%%%%%%%%%%%%%%%%%%%%%%%%%%%%%%%%%
%%%%%%%%%%%%%%%%%%%%%%%%%%%%%%%%%%%%%%%%%%%%%%%%%
\section{Preliminaries on the Hitchin system}\label{hitc.syst}
%%%%%%%%%%%%%%%%%%%%%%%%%%%%%%%%%%%%%%%%%%%%%%%%%
%%%%%%%%%%%%%%%%%%%%%%%%%%%%%%%%%%%%%%%%%%%%%%%%%
We begin by recalling some parts of the theory of $\SL(2,\CC)$ Higgs bundles, developed initially in Hitchin in~\cite{hi87a} and subsequently 
extended by very many authors. The moduli space of stable Higgs bundles carries a rich geometric structure, including a natural hyperk\"ahler 
structure arising from its gauge theoretic interpretation as a hyperk\"ahler quotient~\cite{hklr87}. It is also an algebraic completely integrable 
system~\cite{hi87a,hi87b}, and hence a dense open set (the so-called regular set) is endowed with a semiflat hyperk\"ahler metric~\cite{fr99}.
We explain all of this now. 
%
%%%%%%%%%%%%%%%%%%%%%%%%%%%%%%%%%%%%%%%%%%%%%%%%%
\subsection{\bf The moduli space of Higgs bundles}
%%%%%%%%%%%%%%%%%%%%%%%%%%%%%%%%%%%%%%%%%%%%%%%%%
Let $X$ be a compact Riemann surface of genus $\gamma\geq2$, $K_X$ its canonical bundle, and $p:E \to X$ a complex rank $2$ vector bundle over $X$. 
A holomorphic structure on $E$ is equivalent to a {\em Cauchy-Riemann operator} $\delb:\Omega^0(E)\to\Omega^{0,1}(E)$, so we think of a holomorphic 
vector bundle as a pair $(E,\delb)$. A {\em Higgs field} $\Phi$ is an element $\Phi\in H^0(X,\End(E)\otimes K_X)$, i.e., a holomorphic section of $\End(E)$ 
twisted by the canonical bundle. An $\SL(2,\CC)$ Higgs bundle is a triple $(E,\delb,\Phi)$ for which the determinant line bundle 
$\det E:=\Lambda^2 E$ is holomorphically trivial, in particular $\deg E = 0$, and the Higgs field $\Phi$ is traceless. Thus, 
with $\End_0(E)$ the bundle of tracefree  endomorphisms of $E$,  $\Phi\in H^0(X,\End_0(E)\otimes K_X)$. In the sequel, a 
{\em Higgs bundle} will always refer to this special situation.   Thus a Higgs bundle  is completely specified by a pair $(\delb,\Phi)$. Throughout,    Higgs bundles are considered exclusively on the fixed complex vector bundle $E$ of degree $0$, which will therefore be suppressed from our notation. 

%Thus a Higgs bundle is completely specified by a pair $(\delb,\Phi)$. 

The special complex gauge group $\mc G^c$ consisting of automorphisms of $E$ of unit determinant acts on Higgs bundles by $(\delb, \Phi) \mapsto 
(g^{-1}\circ\delb\circ g,g^{-1}\Phi g)$.  The quotient by this action is not well-behaved unless restricted to the subset of {\em stable} Higgs bundles. 
When $\deg E$ vanishes, a Higgs bundle $(\delb,\Phi)$ is called stable if any $\Phi$-invariant subbundle $L$, i.e., one for which $\Phi(L)\subset L\otimes K_X$, 
has $\deg L < 0$.  Note that if $\delb$ is stable in the usual sense, then $(\delb,\Phi)$ is a stable Higgs bundle for any choice of $\Phi$. We call 
\[
\mc M=\{\mbox{stable Higgs bundles}\}/\mc G^c
\]
the {\em moduli space of Higgs bundles}. This is a smooth complex manifold of dimension $6(\gamma-1)$. Furthermore, if $\mc N$ denotes the 
(smooth quasi-projective manifold) of stable holomorphic structures on $E$, then $T^*\mc N$ embeds as an open dense subset of $\mc M$. The tangent space to $\mathcal M$ at an equivalence class $[(\delb,\Phi)]$ fits into the exact sequence~\cite{ni91}
\begin{multline*}
H^0(\End_0(E))\longrightarrow H^0(\End_0(E)\otimes K_X)\longrightarrow T_{[(\delb,\Phi)]}\mc M \\ 
\longrightarrow H^1(\End_0(E))\longrightarrow H^1(\End_0(E)\otimes K_X).
\end{multline*}
We use here the abbreviated notation $H^j(F)$ for $H^j(X, F)$. The holomorphic structure on $\End_0(E)$ is inherited from the one on $E$, and the maps 
$H^j(\End_0(E))\to H^j(\End_0(E)\otimes K_X)$ are induced by $[\Phi,\cdot]$ acting on the sheaf of holomorphic sections of $\End_0(E)$.
The restriction of 
the natural nondegenerate pairing $H^0(\End_0(E)\otimes K_X)\times H^1(\End_0(E))\to\C$ coming from Serre duality gives rise to a {\em holomorphic 
symplectic form} $\eta$ on $\mc M$ which extends the natural complex symplectic form of $T^*\mc N$. Note also that $H^0(\End_0(E))\cong 
H^1 (\End_0(E) \otimes K_X) = 0$ if $\delb$ is stable. 
%
%%%%%%%%%%%%%%%%%%%%%%%%%%%%%%%%%%%%%%%%%%%%%%%%%
\subsection{\bf Algebraic integrable systems}\label{subsect:algebrintsyst}
%%%%%%%%%%%%%%%%%%%%%%%%%%%%%%%%%%%%%%%%%%%%%%%%%
We next exhibit on the complex symplectic manifold $(\mc M,\eta)$ the structure of an {\em algebraic integrable system}~\cite{hi87a,hi87b}. 
Let $\mc B=H^0(K_X^2)$ denote the space of holomorphic quadratic differentials, and $\Lambda \subset \mc B$ the discriminant locus, consisting
of holomorphic quadratic differentials for which at least one zero is not simple. This is a closed subvariety which is invariant under the multiplicative 
action of $\CC^\times$, and hence $\mc B':=\mc B\setminus \Lambda$ is an open dense subset of $\mc B$. 

The determinant is invariant under conjugation, hence descends to a holomorphic map
\[
\det:\mc M\to\mc B,\qquad[(\delb,\Phi)]\mapsto\det\Phi,
\]
called the {\em Hitchin fibration}~\cite{hi87a}. This map is proper and surjective.  It can be shown that there exist $3(\gamma-3)$ linearly
independent functions on $\mc M':=\det^{-1}(\mc B')$ which commute with respect to the Poisson bracket corresponding to the holomorphic
symplectic form $\eta$. Hence, $\mc M'$ is a completely integrable system over this set of regular values, cf.\ \cite[Section 44]{gs90}
and~\cite{fr99}. In particular, generic
fibers of $\det$ are affine tori. Identifying $T^*_q\mc B'$ with the invariant vector fields on $\mc M'_q$ yields a transitive action 
on the fibers by taking the time-$1$ map of the flow generated by these vector fields. The kernel 
$\Gamma_q$ is a full rank lattice in $T_q^*\mc B'$ (i.e., its $\R$-linear span equals $T_q^*\mc B'$), and
$\Gamma=\bigcup_{q\in\mc B'}\Gamma_q$ is a local system over $\mc B'$. This gives an analytic family of complex
tori $\mc A=T^*\mc B'/\Gamma$. Since $\Gamma$ is complex Lagrangian for the holomorphic symplectic form $\omega_{T^*\mc B'}$,
this form descends to a holomorphic symplectic form $\hat{\eta}$ on $\mc A$. 

We now and henceforth fix a holomorphic square root 
\begin{equation*}
\Theta=K_X^{1/2}
\label{canonicaltheta}
\end{equation*} 
of the canonical bundle. We then define the {\em Hitchin section} of $\mc M$ by
\[
  \mc H:\mc B\to\mc M,\quad\mc H(q)=\left[(\delb_{\Theta \oplus \Theta^*} ,\Phi_q) \right], \quad
  \mbox{where}\ \ \Phi_q = \begin{pmatrix}0&-q\\1&0\end{pmatrix}. 
\]
Then $\mc H(\mc B')$ is complex Lagrangian: $\mc H^*\eta=0$, since only $\Phi$ varies. This gives a local symplectomorphism between $(T^*\mc B',\omega_{T^*\mc B'})$ and $(\mc M',\eta)$.
On each fiber, this is the Albanese mapping determined by the point $\mc H(q)\in\mc M'_q$.   We must also
identify the affine complex torus $\mc M'_q$ algebraically; 
this turns out to be a subvariety of the Jacobian of the related Riemann surface 
\[
S_q=\{\alpha\in K_X\mid\alpha^2=q(p(\alpha))\} \subset K_X.
\]
called the spectral curve associated to $q$. Since the zeroes of $q$ are simple, $p_q:=p|_{S_q}: S_q \to X$ is a twofold covering
between smooth curves with simple branch points at the zeroes of $q$, hence by the Riemann-Hurwitz formula, $S_q$ has genus
$4\gamma-3$. We think of points of $S_q$ as the eigenvalues of $\Phi$ (this explains the name {\em spectral curve}).

We summarize this discussion in the
\begin{proposition}
    There is a symplectomorphism between $(\mc M',\eta)$ and $(\mc A,\hat{\eta})$ which intertwines the $\CC^\times$
    action on the two spaces. 
  \label{firstdiffeo}
\end{proposition}
\begin{remark}
  Note that the implicit $\CC^\times$ action on $T^* \calB'$ here is not the standard pullback action. The one here dilates the base but
  acts trivially on the fibers. 
  Another important fact is that the $\CC^\times$ action identifies the fibers $\calM_q'$ and $\calM_{t^2 q}'$ 
for every $t \in \CC^\times$.
\end{remark}

There is a more intrinsic description of this picture using the holomorphic Liouville form $\lambda\in\Omega^1(K_X)$, $\lambda_\alpha(v)=\alpha(p_*v)$ for any 
$\alpha\in K_X$, $v\in T_\alpha K_X$. Its pullback by the inclusion map $\iota_q:S_q\to K_X$ is the {\em Seiberg-Witten differential} on $S_q$,
\[
\lambda_{\mr{SW}}(q):=\iota_q^*\lambda\in H^0(K_{S_q}) \cong H^{1,0}(S_q),
\]
which in particular is a closed form.
If $q$ is clear from the context, we simply write $\lambda_{\mr{SW}}$. 
Now denote by $\sigma_q$ the involution of $S_q$ obtained by restricting the map $\sigma$ which is multiplication by $-1$ on the fibers of $K_X$. 
Then $\sigma^*_q(\pm\lambda_{\mr{SW}}(q)) 
=\mp\lambda_{\mr{SW}}(q)$ are the two ``eigenforms'' of $p_q^*\Phi:p_q^*E\to p_q^*E\otimes p_q^*K_X$. The two corresponding holomorphic line eigenbundles 
$L_\pm$ of $p_q^*E$ are interchanged under $\sigma_q$. Since $L_+\otimes L_-\cong p_q^*K_X^{-1}$ we see that $\sigma^*_qL_+\cong L_+^{-1}\otimes p^*_qK_X^{-1}$. 
Twisting by $\Theta_q=p_q^*\Theta$ we see that $\sigma_q(L_+\otimes\Theta_q)=(L_+\otimes\Theta_q)^{-1}$, i.e., $L_+\otimes\Theta_q$ lies in what we call 
the {\em Prym-Picard variety} $\mr{PPrym}(S_q)=\{L\in\mr{Pic}(S_q)\mid\sigma^*L=L^*\}$. 

Summarizing, any Higgs bundle $(\delb,\Phi)$ with $\det\Phi\in\mc B'$ induces a pair $(S_q,L_+)$ with $L_+\otimes\Theta_q\in\mr{PPrym}(S_q)$. Conversely, 
$(\delb,\Phi)$ with $q=\det\Phi\in\mc B'$ can be recovered from a line bundle in $\mr{PPrym}(S_q)$. Consequently, the choice of square root $\Theta_q = K_X^{1/2}$ 
identifies $\mc M'_q$ biholomorphically with $\mr{PPrym}(S_q)$. This, in turn, gets identified via the Hitchin section with its Albanese variety 
$H^0(K_{\mr{PPrym}(S_q)})^*/H_1(\mr{PPrym}(S_q);\Z)$. This shows that $\mc M'\to\mc B'$ is an {\em algebraic integrable system}.
%
%%%%%%%%%%%%%%%%%%%%%%%%%%%%%%%%%%%%%%%%%%%%%%%%%
\subsection{\bf The special K\"ahler metric}\label{spe.kah.met}
%%%%%%%%%%%%%%%%%%%%%%%%%%%%%%%%%%%%%%%%%%%%%%%%%

A K\"ahler manifold $(M^{2m},\omega,I)$ is called {\em special K\"ahler} if there exists a flat, symplectic, torsionfree connection $\nabla$ 
such that, regarding $I$ as a $TM$-valued $1$-form, $d_\nabla I=0$.  The basic reference for special K\"ahler metrics is \cite{fr99}, and 
see \cite{hhp10} for the case of Hitchin systems. 

The analytic family of spectral curves $\mc S=\bigcup_{q\in\mc B'}S_q\to\mc B'$ induces a special K\"ahler metric on $\mc B'$.   To see 
this, first identify the Albanese varieties of the previous section with 
\[
\Prym(S_q):=H^0(K_{S_q})_\odd^*/H_1(S_q;\Z)_\odd
\]
where $H^0(K_{S_q})_\odd$ and $H_1(S_q;\Z)_\odd$ denote the $(-1)$-eigenspaces of $H^0(K_{S_q})$ and $H_1(S_q;\Z)$ under the involution 
$\sigma$, cf.\ \cite[Proposition 12.4.2]{bl04}. Moreover, considering $\mc B'$ as the $\sigma$-invariant deformation space of a given 
spectral curve $S_q$, we have $T_q\mc B'\cong H^0(N_{S_q})_\odd\cong H^0(K_{S_q})_\odd$ where the canonical symplectic form $d\lambda$ 
on $K_X$ is used to identify the normal bundle $N_{S_q}$ of $S_q$ with the canonical bundle of $K_{S_q}$ (cf.\ also~\cite{ba06,hhp10}). 
It follows that $T^*_q\mc B'\cong H^0(K_{S_q})_\odd^*\cong\C^{3\gamma-3}$. This contains the integer lattice $\Gamma_q=H_1(S_q;\Z)_\odd
\cong\Z^{6\gamma-6}$. Since $H_1(S_q;\Z)_\odd\cong H_1(\Prym(S_q);\Z)$, we can choose a symplectic basis for the intersection form, 
$\alpha_1(q),\ldots,\alpha_m(q),\beta_1(q),\ldots,\beta_m(q)$, $m=3\gamma-3$, in $\Gamma_q$. This intersection form (the polarization 
of $\Prym(S_q)$) is twice the restriction of the intersection form of $S_q$ (the canonical polarization of the {\em Jacobian} of $S_q$), 
cf.\ \cite[p.\ 377]{bl04}. 

An important feature of any special K\"ahler metric is the existence of {\em conjugate} coordinate systems 
$(z_1,\ldots,z_m)$ and $(w_1,\ldots,w_m)$, i.e., holomorphic  coordinates such that
$(x_1,\ldots,x_m,y_1,\ldots,y_m)$, where $\Re(z_i)=x_i$ and $\Re(w_i)=-y_i$, are Darboux coordinates for $\omega$. 
The local system $\Gamma= \bigcup_{q\in\mc B'}\Gamma_q$ is spanned locally by differentials of 
Darboux coordinates $(dx_i,dy_i)$ and induces a real, torsionfree, flat symplectic connection $\nabla$ over $\mc B'$
by declaring  $\nabla dx_i=\nabla dy_i=0$ for $i=1,\ldots, m$. Thus we can choose 
the coordinates $(x_i,y_i)$ in such a way that conjugate holomorphic coordinates are 
\begin{equation}
z_i(q)=\int_{\alpha_i(q)}\lambda_{\mr SW}(q),\quad w_i(q)=\int_{\beta_i(q)}\lambda_{\mr SW}(q),\quad i=1,\ldots,m,
\label{conjugated}
\end{equation}
\cite[Proof of Theorem 3.4]{fr99}. In terms of these, the K\"ahler form equals
\[
\omega_{\mr{sK}}=\sum_{i=1}^{3\gamma-3}dx_i\wedge dy_i=-\frac 14\sum_i(dz_i\wedge d\bar w_i+d\bar z_i\wedge dw_i). 
\]

There is an alternate and quite explicit expression for $\omega_{\rm sK}$. To this end, observe that 
\[
dz_i(\dot q) = \int_{\alpha_i(q)}\nabla^{\mr GM}_{\dot q}\lambda_{\mr SW}, \quad dw_i(\dot q) = \int_{\beta_i(q)}\nabla^{\mr GM}_{\dot q}
\lambda_{\mr SW},\quad i=1,\ldots,m, 
\]
where $\nabla^{\mr GM}$ is the Gau{\ss}-Manin connection and $\lambda_{\mr SW}:\mc B'\to\bigcup_{q\in\mc B'} H^{1,0}(S_q)$ is considered as a section. 
Then $\nabla^{\mr GM}_{\dot q}\lambda_\SW$ is the contraction of $d\lambda_{\mr SW}$ by the normal vector field $N_{\dot q}$ corresponding 
to $\dot q$. By Proposition 1 in \cite{dh75} (cf.\ also Proposition 8.2 in \cite{hhp10}) we have 
\begin{equation}\label{deriv_sw_diff}
\nabla^{\mr GM}_{\dot q}\lambda_\SW= \frac 12 \tau_{\dot q}
\end{equation}
where $\tau_{\dot q}$ is the holomorphic 1-form on $S_q$ corresponding to $\dot q$ under the isomorphism 
\begin{equation}\label{identif_tangent_space}
T_q \mathcal{B}'= 
H^0(K_X^2)
\overset{\cong}{\longrightarrow} H^0(K_{S_q})_\odd, \quad  \dot q \mapsto \tau_{\dot q} := \frac{\dot q}{\lambda_{\mr SW}}.  
\end{equation}
There is a seemingly anomalous factor of $\frac{1}{2}$ here compared to the cited formula in \cite{dh75}. The reason 
is that  their expression $\alpha_{\dot q}$ which appears in the right hand side of
their formula for the Gau{\ss}-Manin derivative of $\lambda_{\SW}$ is precisely $1/2$ of $\tau_{\dot q}$ as we have
defined it here. 

\begin{remark}\label{tauq}
The special case where $\dot q = q$ is of particular interest since it generates the $\CC^\times$ action
on $\calB'$.  (Recall however that we work only with the $\RR^+$ action.) For this infinitesimal variation,
we have $\tau_{q}=\lambda_{\mathrm{SW}}$ and hence
\[
\nabla^{\mr GM}_{q}\lambda_{\mathrm{SW}}= \frac{1}{2}\lambda_{\mathrm{SW}}.
\]
%This is consistent with \cite [Corollary 8.3]{hhp10}, again relative to our definition of $\tau_{\dot q}$. 
\end{remark}

The associated K\"ahler metric $g_{\mr{sK}}(\dot q,\dot q)$ equals $\omega_{\mathrm{sK}}(\dot q,I\dot q)$ for the constant complex 
structure $I= i$. It is therefore given by 
\begin{align*}
g_{\mathrm{sK}}(\dot q, \dot q) &= \frac{i}{2} \sum_j \left(dz_{j}(\dot q) d \bar w_{j}(\dot q) - dw_{j}(\dot q) d \bar z_{j}(\dot q)\right) \\
&= \frac {i}{2}\sum_j\int_{\alpha_j}\nabla^{\mr GM}_{\dot q}\lambda_{\mr SW}\int_{\beta_j}\nabla^{\mr GM}_{\dot q}\bar\lambda_{\mr SW}-
\int_{\beta_j}\nabla^{\mr GM}_{\dot q}\lambda_{\mr SW}\int_{\alpha_j}\nabla^{\mr GM}_{\dot q}\bar\lambda_{\mr SW}\\
&= \frac{i}{8} \sum_j \int_{\alpha_j} \tau_{\dot q} \int_{\beta_j} \bar \tau_{\dot q} - \int_{\beta_j} \tau_{\dot q} \int_{\alpha_j} \bar \tau_{\dot q}\\
&= \frac{i}{8} \int_{S_q} \tau_{\dot q} \wedge \bar \tau_{\dot q} = \frac{1}{8} \int_{S_q} |\tau_{\dot q}|^2\, dA,
\end{align*}
where we have used the Riemann bilinear relations.  Here $dA$ is the area form on $S_q$ induced from the one on $X$ 
for any metric in the given conformal class on $X$ and we recall that the quantity $|\alpha|^2 dA$ is conformally invariant 
when $\alpha$ is a $1$-form. Note also that $\int_c\lambda_{\mr SW}$ vanishes for any even cycle $c$, since $\lambda_{\mr SW}$ is 
odd with respect to $\sigma$. This identifies the special K\"ahler metric on $T_q \mathcal{B}'$ with an eighth of the natural $L^2$-metric 
 \[
\|\alpha\|_{L^2}^2 = i \int_{S_q} \alpha \wedge \bar \alpha = \int_{S_q} |\alpha|^2\, dA,
\]
on $H^0(K_{S_q})_\odd$ via the isomorphism $\dot q \mapsto \tau_{\dot q}$.  Using $\tau_{\dot q}=\dot q/\lambda_{\mathrm{SW}}$ and 
$\lambda_{\mathrm{SW}}^2 = q$ we obtain that $|\tau_{\dot q}|^2 =  |\dot q|^2/|q|$ and so the last integral  may be converted into an 
integral over the base Riemann surface:
\begin{equation}
g_{\mathrm{sK}}(\dot q, \dot q)  =\frac{1}{8} \int_{S_q} |\tau_{\dot q}|^2\, dA = \frac 18 \int_{S_q} \frac{|\dot q|^2}{|q|} \, dA = 
\frac 14 \int_{X} \frac{|\dot q|^2}{|q|} \, dA
\label{diffexpsk}
\end{equation}
This representation of the special K\"ahler metric will be 
important later. 
For any holomorphic quadratic differential $q$, the quantity $|q|\, dA$ is conformally invariant, so again the
  choice of metric in the conformal class is irrelevant.
We single out one key consequence of the preceding discussion.
\begin{corollary}\label{smoothsK}
The special K\"ahler metric $g_{\sK}$ depends smoothly on the basepoint $q \in \calB'$. 
\end{corollary}

\begin{proof}
This may be seen from the following local coordinate expression for $\tau_{\dot q}$.
In a local holomorphic coordinate chart, $q(z)=f(z)dz^2$ and $\dot q(z)=\dot f(z)dz^2$, and since $z=0$ is a simple zero of $q$,
$f(0)=0$ but $f'(0) \neq 0$. Let $(z,w)$ be canonical local coordinates on $K_X$, so $\lambda_{\mr SW}=wdz$. Then $S_q = 
\{w^2=f(z)\}$ and hence 
\[
2wdw=f'(z)dz 
\]
there. In particular, $\lambda_\SW=2w^2dw/f'(z)$ and $\dot q = 4w^2 \dot{f}(z) dw^2/ f'(z)^2$, so 
\[
\tau_{\dot q} = \frac{\dot q}{\lambda_\SW}= \frac{2 \dot f(z)}{f'(z)} dw.
\]
This computation shows that $\tau_{\dot q}$ and hence $g_{\sK}$   depends smoothly on $q$.
\end{proof}

Note that the smoothness asserted in the corollary is not immediately apparent from some of the other expressions, e.g.\ the final one in \eqref{diffexpsk}.

We conclude this section by discussing the conic structure of this metric. Recall the $\C^\times$-action on $\mathcal{B}'$:
\[
\varphi_\lambda(q):=\lambda^2 q, \quad q \in \mathcal{B}', \lambda \in \C^\times.
\]
It is immediate from \eqref{conjugated} and the defining relation $\lambda_{\SW}^2 = q$ on $S_q$ that the coordinates $z_i$ and $w_i$ 
are homogeneous of degree $1$, i.e.
\[
z_i(\varphi_\lambda(q)) = \int_{\alpha_i} \tau_{\lambda q} = \lambda z_i(q), \quad
w_i(\varphi_\lambda(q)) = \int_{\beta_i} \tau_{\lambda q} = \lambda w_i(q).
\]
for $\lambda \in \calW$, where $\calW$ is a neighborhood of $1 \in \C^\times$.
Euler's formula for the derivative of homogeneous functions now gives that $\sum_i z_i \partial w_j/\partial z_i = w_j$, hence 
\[
\mathcal{F}(q) = \frac{1}{2} \sum_j z_jw_j. 
\]
defines a holomorphic prepotential. Indeed, since $\partial w_i/\partial z_j = \partial w_j/\partial z_i$ we get
\[
\partial \mathcal{F}/\partial z_j = \tfrac 12 (w_j + \sum_i z_i \partial w_i/\partial z_j) = \tfrac 12 (w_j + \sum_i z_i \partial w_j/\partial z_i)= w_j.
\]
This holomorphic prepotential is of course homogeneous of degree $2$, i.e.\ $\mathcal{F}(\varphi_\lambda(q))= \lambda^2\mathcal{F}(q)$. This establishes %the sector inside $\mathcal{B}'$ where $z_i$ and $w_i$ are defined is 
%a {\em conic special K\"ahler domain} and thus 
$\mathcal{B}'$ as a {\em conic special K\"ahler manifold}, see Proposition 6  
in \cite{cm09}. 

Computing locally again, we find using the Riemann bilinear relations and the relation $\tau_q^2=q$ that the K\"ahler potential is given by 
\begin{align*}
K(q) &= \frac 12 \Im \sum_j w_j\bar z_j = \frac i4 \sum_j (z_j\bar w_j- \bar z_j w_j)\\
&= \frac i4 \sum_j \int_{\alpha_j} \tau_q \int_{\beta_j} \bar \tau_q -  \int_{\alpha_j} \bar\tau_q \int_{\beta_j} \tau_q\\
&= \frac i4 \int_{S_q} \tau_q \wedge \bar \tau_q = \frac 14 \int_{S_q} |\tau_q|^2\, dA = \frac 12 \int_X |q|\, dA.
\end{align*}
Let $\mathcal{S}'=\{ q \in \mathcal{B}' : \int_X |q| \, dA =1\}$ the $L^1$-unit sphere in $\mathcal{B}'$. By Corollary 4 in \cite{bc03}, we find that
\begin{equation}\label{isometryphi}
\phi : (\R_+ \times \mathcal{S}',dt^2 + t^2 g_{\sK}|_{\mathcal{S}'}) \to (\mathcal{B}', g_\sK), \quad (t,q) \mapsto t^2q
\end{equation}
is an isometry. This establishes that $\mathcal{B'}$ is a metric cone. In particular, for $q \in \mathcal{B}'$ with  $\int_X |q| \, dA =1$ the curve $t \mapsto t^2 q$ is a unit speed geodesic. As a check on this, observe that
\begin{equation}\label{differential}
d\phi|_{(t, q)}  (\del_t) = 2t q, \qquad d\phi|_{(t,q)} (\dot q) = t^2 \dot q.
\end{equation}
On the other hand,
\begin{multline*}
g_{\sK}(\dot q, \dot q)|_{t^2 q} = \frac{i}{8}\int_{S_{t^2 q} } (\dot q/ \lambda_{\SW})  \wedge \overline{(\dot q/ \lambda_{\SW})} 
\\ =  \frac{i}{8t^2} \int_{S_q} (\dot q/ \lambda_{\SW})  \wedge \overline{\dot q/\lambda_{\SW}} = \frac{1}{t^2} g_{\sK}(\dot q, \dot q)|_q,
\end{multline*}
so
\begin{equation}
(\| 2t q\|_{\sK}^2)|_{ t^2 q} = 4(\|q\|_{\sK}^2)|_q=1, \qquad (\|t^2 \dot q\|_{\sK}^2)|_{t^2 q} = t^2 (\|\dot q\|^2_{\sK})|_q.
\label{homogeneity}
\end{equation}
Here we have used that $(\|q\|^2_\sK)|_q = \frac 14 \int_X |q| \, dA = \frac 14$ for $q \in \mathcal{S}'$. Thus Equations \eqref{differential} and \eqref{homogeneity} indeed reconfirm the conic structure of $g_\sK$.

%%%%%%%%%%%%%%%%%%%%%%%%%%%%%%%%%%%%%%%%%%%%%%%%%
\subsection{\bf Hyperk\"ahler metrics}\label{subsect:hkmetrics}
%%%%%%%%%%%%%%%%%%%%%%%%%%%%%%%%%%%%%%%%%%%%%%%%%
A Riemannian manifold $(M,g)$ is called {\em hyperk\"ahler} if it carries three integrable complex structures $I$, $J$ and $K$ which 
satisfy the quaternion algebra relations and such that the associated $2$-forms $\omega_C(\cdot,\,\cdot)=g(\cdot,\,C\cdot)$,
$C = I, J, K$, are each closed. In particular, every specialization $(M,C,\omega_C)$ is K\"ahler (this is also true when
$C = aI + bJ + cK$ where $a,b,c$ are constants with $a^2 + b^2 + c^2 = 1$), whence the name hyperk\"ahler. 
The two examples of hyperk\"ahler metrics of 
interest here are the {\em Hitchin metric} on $\mc M$ and the {\em semiflat metric} on $\mc M'$.

%\medskip

\subsubsection{Semiflat metric}
If $(M,\omega,\nabla)$ is any manifold with a special K\"ahler structure, with K\"ahler metric $g_{\mr sK}$, then $T^*M$ carries a natural 
semiflat hyperk\"ahler metric $g_{\sfl}$, cf.~\cite[Theorem 2.1]{fr99}.  The name semiflat comes from the fact that $g_{\sfl}$ is flat on each
fiber of $T^*M$. In particular, if $\Gamma$ is a local system in $T^*M$ of full rank, then $g_{\sfl}$ pushes down to a semiflat metric on the 
torus bundle $T^*M /\Gamma$. We consider this in the special case $M=\mc B'$, where $\mathcal A=T^* \mc B' / \Gamma
  \cong \mc M'$, the analytic family $\mathcal A$ of complex tori introduced in \textsection \ref{subsect:algebrintsyst}.  The existence 
of such a metric is common to any algebraic integrable system, ~\cite[Theorem 3.8]{fr99}. 

To construct $g_{\sfl}$, note that the connection $\nabla$ induces a distribution of horizontal and complex subspaces of $T^*M$. Then,
 relative to the decomposition $T_\alpha T^*M\cong T_{\pi(\alpha)}M \oplus T_{\pi(\alpha)}^*M$, $g_{\sfl}$ equals $g_{\pi(\alpha)}\oplus
g^{-1}_{\pi(\alpha)}$; the integrability is ensured by the differential geometric conditions on a special K\"ahler metric. It is clearly flat 
in the fiber directions. In local coordinates $(x_i,y_i,p_i,q_i)$ of $T^*M$ induced by Darboux coordinates $(x_i,y_i)$ for $\omega$, 
the K\"ahler form $\omega_I$ for the natural complex structure on $T^*M$ is 
\[
\omega_I=\sum_idx_i\wedge dy_i+dp_i\wedge dq_i.
\]
As noted earlier, if $M=\mc B'$, then $g_{\sfl}$ descends to the quotient $\mc A=T^*\mc B'/\Lambda$, and thus induces a metric
on $\mc M'$ which we still denote by $g_{\sfl}$. The invariant vector fields on the fibers of $\mc M'$ are given by the $\eta$-Hamiltonian 
vector fields $X_f$ of functions $f\circ\pi$ where $f$ is a locally defined function on $\mc B'$ (see for instance~\cite[(44.5)]{gs90}). 
Hence, if $X_f$ is a vector field on $\mc M'$ tangent to the fibers, then 
\[
g_{\sfl}(X_f,X_f)=g_{\rm{sK}}^{-1}(df,df).
\]
Computing the dual metric $g_\sK^{-1}$ on $T^*_q\mathcal{B}'$ amounts to computing the metric on $H^0(K_{S_q})^*_\odd$ dual to 
the $L^2$-metric on $H^0(K_{S_q})_\odd$. The complex antilinear isomorphim $H^0(K_{S_q})^* \to H^0(K_{S_q})$ obtained by dualizing 
with respect to the $L^2$-metric simply is the composition
\[
H^0(K_{S_q})^* = \mathcal{H}^{1,0}(S_q)^* \longrightarrow \mathcal{H}^{0,1}(S_q) \longrightarrow \mathcal{H}^{1,0}(S_q)=H^0(K_{S_q}),
\]
where the first arrow is given by Serre duality and the second one by complex conjugation $\bar{\quad}:\mathcal{H}^{0,1}(S_q) \to 
\mathcal{H}^{1,0}(S_q)$, exchanging the space of anti-holomorphic and holomorphic forms. So if $df(q)$ is dual to $\alpha 
\in H^0(K_{S_q})_{\mathrm{odd}}$ then
\[
g_\sK^{-1}(df(q),df(q)) = \int_{S_q} |\alpha|^2 \, dA=:g_\sfl(\alpha,\alpha).
\]
This shows that the vertical part of the semiflat metric is the natural $L^2$-metric on $\Prym(S_q)$. We return to this fact in Section 3.

 We also wish to describe the Prym variety in terms of unitary data. In fact, each line bundle $L$ in $\Prym(S_q)$ corresponds to an {\it odd} 
flat unitary connection  on the trivial complex line bundle.  In other words, $L$ is represented by a connection 1-form $\eta \in 
\Omega^1(S_q,i \R)$ such that $d \eta =0$ and $\sigma^* \eta = - \eta$. This space is acted on by odd gauge transformations, i.e., 
maps $g: S_q \to S^1$ such that $g \circ \sigma = g^{-1}$. We obtain
\[
\Prym(S_q) = \frac{H^1(S_q; i \R)_\odd}{H^1_\Z(S_q;i\R)_\odd}.
\]
If $\eta \in \mathcal{H}^1(S_q, i\R)_\odd$ is a harmonic representative of a class in $H^1(S_q; i\R)_{\odd}$, then $\eta=\alpha - \bar \alpha$ 
for $\alpha=\eta^{1,0} \in H^0(K_{S_q})_\odd$. Here we have used that $\mathcal{H}^1(S_q,\C)=\mathcal{H}^{1,0}(S_q)\oplus 
\mathcal{H}^{0,1}(S_q)$. So finally
\begin{equation}\label{vertical_L2}
g_{\sfl}(\eta,\eta):=g_{\sfl}(\alpha,\alpha)=\frac 12 \int_{S_q}|\eta|^2\, dA = \int_X |\eta|^2 \, dA, 
\end{equation}
which is the form of the metric we will use from now on. In Section \ref{L2_metric_limit_conf} we will reinterpret the space of imaginary 
odd harmonic 1-forms on $S_q$ as a space of $L^2$-harmonic forms with values in a twisted line bundle on the punctured base 
Riemann surface $X^\times$, reducing the $L^2$-integral over $S_q$ to an integral over $X$.

Parallel to Corollary~\ref{smoothsK}  and its proof we have
\begin{corollary}\label{smoothsf}
The semiflat metric is smooth on $\calM'$. 
\end{corollary}

%\medskip

\subsubsection{Hitchin metric}
The second hyperk\"ahler metric we consider is defined on all of $\mc M$ and stems from a gauge-theoretic reinterpretation of $\mc M$.
More concretely, fix a hermitian metric $H$ on $E$. Holomorphic structures $\delb$ are then in $1-1$-correspondence with special 
unitary connections. After the choice of a base connection these correspond to elements in $\Omega^{0,1}(\mf{sl}(E))$. For such an 
endomorphism valued form $A$ we denote the corresponding Cauchy-Riemann operator by $\delb_A$. The condition 
$\Phi\in H^0(X,\mf{sl}(E)\otimes K_X)$ is equivalent to $\delb_A\Phi=0$, where $\Phi$ is regarded as a section of $\Lambda^{1,0}T^*X 
\otimes\mf{sl}(E)$. In particular, we get an induced $\mc G^c$-action on $(A,\Phi)$. We denote this action by $(A^g,\Phi^g)$ for 
$g\in\mc G^c$. Hitchin~\cite{hi87a} proves that in the $\mc G^c$-equivalence class $[E,\delb,\Phi]=[A,\Phi]$ there exists a
representative $(A^g,\Phi^g)$ unique up to special unitary gauge transformations such that the so-called {\em self-duality equations} or {\em Hitchin equations}
 (with respect to $H$)
\begin{equation}\label{sel.dua.equ}
\mu(A,\Phi):=(F_A+[\Phi\wedge\Phi^*],\delb_A\Phi)=0
\end{equation}
hold. Here, $F_A$ denotes the curvature of $A$ and $\Phi^*$ is the hermitian conjugate; we refer to $\mu$ as the hyperk\"ahler moment  map. 

\begin{remark}
Alternatively, we can fix a Higgs bundle $(\delb,\Phi)$ and ask for a hermitian metric $H$ such that $F_H+[\Phi\wedge\Phi^{*_H}]=0$ 
where $*_H$ is the adjoint taken with respect to $H$ and $F_H$ is the curvature of the Chern connection $A$. The pair $(A,\Phi)$ is 
then a solution to the self-duality equation with respect to $H$.
\end{remark}

Stability of $(E,\Phi)$ translates into the irreducibility of $(A,\Phi)$. If $\mc G$ denotes the special unitary gauge group it follows that
\[
\mc M\cong\{(A,\Phi) \in\Omega^{1}(\mf{su}(E))\times\Omega^{1,0}(\mf{sl}(E))\mbox{ irreducible solves~\eqref{sel.dua.equ}}\}/\mc G.
\]
The  map $\mu$ %equations~\eqref{sel.dua.equ} 
can be interpreted as a {\em hyperk\"ahler moment map} with respect to the natural action of the 
special unitary gauge group $\mc G$ on the quaternionic vector space $\Omega^{0,1}(\mf{sl}(E))\times\Omega^{1,0}(\mf{sl}(E))$ with  
its natural flat hyperk\"ahler metric
\[
\|(\alpha, \varphi)\|_{L^2}^2 = 2i\int_X \Tr (\alpha^*\! \wedge \alpha + \varphi \wedge \varphi^*)
\]
(note that $\Omega^1(\mf{su}(E))\cong\Omega^{0,1}(\mf{sl}(E))$).
Consequently, this metric descends to a hyperk\"ahler metric on the quotient $\mc M$~\cite{hklr87}. We 
describe this metric next. Let $\mf{su}(E)$ denote the  tracefree endomorphisms of $E$, which are skew-hermitian with respect to the hermitian metric $H$ fixed above.  We endow $\mathfrak{sl}(E) $ with the hermitian inner product given by  $\langle A,B\rangle =\Tr(AB^{\ast})$ and extend it to $\mf{sl}(E)$-valued forms by choosing a conformal background metric on $X$. Fix a configuration $(A,\Phi)$ 
and consider the deformation complex
\begin{multline*}
0\to\Omega^0(\su(E)) \xrightarrow{D^1_{(A,\Phi)}}\Omega^1(\su(E))\oplus\Omega^{1,0}(\mf{sl}(E)) \\
\xrightarrow{D^2_{(A,\Phi)}}\Omega^2(\su(E))\oplus\Omega^2(\mf{sl}(E)) \to 0
\end{multline*} 
 The first differential 
\[
D^1_{(A,\Phi)}(\gamma)=(d_A \gamma,[\Phi\wedge\gamma]),
\]
is the linearized action of $\mc G$ at $(A,\Phi)$, while the second is the linearization of the hyperk\"ahler moment map,
\[
D^2_{(A,\Phi)}(\dot A,\dot\Phi)=(d_A\dot A+[\dot\Phi\wedge\Phi^*]+[\Phi\wedge\dot\Phi^*],\delb_A\dot\Phi+[\dot A,\Phi]).
\]
The tangent space to $\mc M$ at $[A,\Phi]$ is then identified with the quotient 
\[
\ker D^2_{(A,\Phi)}/\im D^1_{(A,\Phi)}\cong\ker D^2_{(A,\Phi)}\cap(\im D^1_{(A,\Phi)})^\perp.
\] 
  Then %for a tangent vector $(\dot A, \dot \Phi)$, 
\[
\int_X\langle d_A\gamma,\dot A\rangle \, dA = \int_X\langle\gamma,d_A^*\dot A\rangle \, dA 
\]
and
\[
\int_X \langle [\Phi \wedge \gamma], \dot \Phi \rangle \, dA = - \int_X \langle \gamma, i\ast \pi^{\skew} [ \Phi^* \!\wedge \dot \Phi ]\rangle \, dA,
\]
where $\pi^\skew : \mathfrak{sl}(E) \to \mathfrak{su}(E)$ is the orthogonal  projection, hence $(\dot A,\dot\Phi)\perp\im D^1_{(A,\Phi)}$ with respect to the $L^2$-metric in \eqref{L2hyperkaehlermetric} below if and only if
\begin{equation}
(D^1_{(A,\Phi)})^* (\dot A, \dot \Phi) = 
d_A^*\dot A-2 \pi^{\skew}(i\ast[ \Phi^*\wedge\dot\Phi])=0.
\label{gaugefixing}
\end{equation}
If this is satisfied, we say that $(\dot A, \dot \Phi)$ is in {\em Coulomb gauge} (in {\em gauge} for short). For tangent vectors $(\dot A_i, \dot \Phi_i)$, $i=1,2$ 
in Coulomb gauge, the induced $L^2$-metric is given by  
%\textcolor{blue}{in terms of the conformally invariant $2$-form $\langle \beta_1,\beta_2\rangle:=\Tr (\beta_1\wedge\beta_2)$ (for $\beta_i\in \Omega^1(\mathfrak{sl}(E))$) on $X$ by}
% in terms of any given choice of a metric in the conformal class on $X$ by
  \begin{equation}
    \begin{split}
  g_{L^2}((\alpha_1,\dot\Phi_1),(\alpha_2,\dot\Phi_2)) & = 2 \int_X \Re \langle\alpha_1,\alpha_2 \rangle +
  \Re \langle\dot\Phi_1,\dot\Phi_2\rangle \; dA, \\
  % g_{L^2}((\dot A_1,\dot\Phi_1),(\dot A_2,\dot\Phi_2))
  & =\int_X \langle \dot A_1,\dot A_2 \rangle +2 \Re \langle\dot\Phi_1,\dot\Phi_2\rangle \; dA,
  \end{split}
  \label{L2hyperkaehlermetric}
  \end{equation}
where $\alpha_i$ denotes the $(0,1)$-part of $\dot A_i$, $i=1,2$ and $dA$ denote the area form of the background metric.

\begin{remark}
There is a similar construction when the determinants of the Higgs bundles are not holomorphically trivial, and it can be shown 
that the $L^2$-metric on the moduli space is complete if the degree of $E$ is odd. 
\end{remark}

\medskip

The first goal of this paper is to show that in a sense to be specified below, the semiflat metric is the asymptotic model for the Hitchin metric.
%
%
%%%%%%%%%%%%%%%%%%%%%%%%%%%%%%%%%%%%%%%%%%%%%%%%%
%%%%%%%%%%%%%%%%%%%%%%%%%%%%%%%%%%%%%%%%%%%%%%%%%
\section{The semiflat metric as $L^2$-metric on limiting configurations}\label{L2_metric_limit_conf}
%%%%%%%%%%%%%%%%%%%%%%%%%%%%%%%%%%%%%%%%%%%%%%%%%
%%%%%%%%%%%%%%%%%%%%%%%%%%%%%%%%%%%%%%%%%%%%%%%%%
Our goal in this section is to understand the semiflat metric on $\mc M'$ as a `formal' $L^2$-metric on the space of limiting configurations. 

%In the following, we assume as always that $\rank(E)=2$, $\deg(E)=0$ (i.e.\ $E$ trivial as $\mathcal {C}^\infty$-bundle)
%Higgs bundles: $\SL(2,\C)$-Higgs bundles, i.e.\ trivial determinant, traceless Higgs field.}
%
%%%%%%%%%%%%%%%%%%%%%%%%%%%%%%%%%%%%%%%%%%%%%%%%%
\subsection{Limiting configurations}\label{sect:limitingconfig}
%%%%%%%%%%%%%%%%%%%%%%%%%%%%%%%%%%%%%%%%%%%%%%%%%

One of the main results in \cite{msww14} is that the degeneration of solutions $(A,\Phi)$ to the self-duality equations as $q = \det \Phi \to \infty$ 
is described in terms of solutions of a decoupled version of the self-duality equations.
\begin{definition}\label{limconfdef1}
Let $H$ be a hermitian metric on $E$ and suppose that $q\in H^0(K_X^2)$ has simple zeroes. Set $X^\times_q=X\setminus q^{-1}(0)$. A {\em limiting configuration} 
for $q$ is a Higgs bundle $(A_\infty, \Phi_\infty)$ over $X_q^\times$, which satisfies the equations
\begin{equation}
F_{A_\infty}=0,\quad[\Phi_\infty\wedge\Phi_\infty^*]=0,\quad\delb_{A_\infty}\Phi_\infty=0 
\label{limconfeqns}
\end{equation}
on $X_q^\times$. We call a Higgs field $\Phi$ which satisfies  $[\Phi_\infty\wedge\Phi_\infty^*]=0$ {\em normal}.
%Thus $A_\infty$ is a flat unitary connection $A_\infty$ and $\Phi_\infty$ is a Higgs field which is everywhere normal and satisfies $\det \Phi_\infty = q$
%where $A_\infty$ is a flat unitary connection $A_\infty$ and $\Phi_\infty$ is a Higgs field which is
%everywhere normal and satisfies $\det \Phi_\infty = q$. This pair thus satisfies the equations
%\begin{equation}
%F_{A_\infty}=0,\quad[\Phi_\infty\wedge\Phi_\infty^*]=0,\quad\delb_{A_\infty}\Phi_\infty=0. 
%\label{limconfeqns}
%\end{equation}
%on $X_q^\times$.
\end{definition}

The unitary gauge group $\mc G$ acts on the space of solutions $(A_\infty, \Phi_\infty)$ to \eqref{limconfeqns}, and we define the moduli space
\[
\mc M_\infty = \{ \mbox{all solutions to \eqref{limconfeqns}} \} / \mc G.
\]
Strictly speaking, we have only considered solutions over differentials $q \in \mc B'$, which correspond to the open subset $\mc M_\infty'$ of this 
moduli space. We refer to \cite{moch15} for the definition and description of limiting configurations over points $q \in \mc B \setminus \mc B'$.

There is some ambiguity in this definition in that we can either divide out by gauge transformations which are smooth across the zeroes of $q$ or
by ones which are singular at these points.  The latter group is more complicated to define because it depends on $q$, and most elements 
in its gauge orbit are singular. However, it is not so unreasonable to consider since, as we discuss later in this section,  tangent vectors 
to $\mc M_\infty$ are `renormalized' to be in $L^2$ by using differentials of such singular gauge transformations.  In the following, we use this definition of the quotient space $\mc M_\infty$.  At the other extreme,  it would have been possible to 
take a view consonant with the original definition of limiting configurations in \cite{msww14}, where each $(A_\infty, \Phi_\infty)$ is assumed to 
take a particular normal form in discs  $\D_p$ around each zero of $q$. This is no restriction because any limiting configuration which is
bounded near the zeroes of $q$ can be put into this form with a (bounded) unitary gauge transformation.  With this restriction, we 
divide out by unitary gauge transformations which equal the identity in each $\D_p$. 

Let us note a few properties of this space.  First, it still possesses a Hitchin fibration $\pi_\infty: \mc M_\infty \to \mc B$, 
$\pi_\infty( (A_\infty, \Phi_\infty)) = \det \Phi_\infty$. A priori, $ \det \Phi_\infty$ is only defined on $X^\times_q$, but is bounded near the punctures, 
hence it extends holomorphically to all of $X$. Second, $\calM_\infty$ has a `semi-conic' structure, $[(A_\infty, \Phi_\infty)] 
\mapsto [(A_\infty, t\Phi_\infty)]$ which dilates the Hitchin base and leaves invariant the Prym variety fibers. 

This space arises as a limit of $\mc M$ in two separate ways.  On the one hand, it is shown in~\cite{msww14} that for any Higgs 
bundle $(A,\Phi)$,  there is a complex gauge transformation $g_\infty$ which is singular at the zeroes of $q$, and is unique up to unitary 
transformations, such 
that  $(A,\Phi)^{g_\infty}$ is a limiting configuration $(A_\infty,\Phi_\infty)$ with $\det\Phi_\infty=\det\Phi$.  Using that $g_\infty$ is the limit 
of smooth complex gauge transformations, one may approximate elements of $\mc M_\infty$ by representatives of sequences of
elements in $\mc M$.  On the 
other hand, consider instead the family of moduli spaces $\mc M_t$ consisting of solutions to the scaled Hitchin equations
\[
 \mu_t(A,\Phi):=(F_A+ t^2 [\Phi\wedge\Phi^*],\delb_A\Phi)=0
\]
modulo unitary gauge transformations. It follows from the main result of \cite{msww14} that away from the discriminant locus,
this family of spaces converges to $\mc M_\infty$, i.e., 
\[
\lim_{t \to \infty} \mc M_t' = \mc M_\infty'.
\]
This is meant in the following sense.  The diffeomorphism $\calF$ described in \eqref{MinftytoM} can be recast 
as a family of natural diffeomorphisms $\calF_t: \calM_\infty' \to \calM_t'$. Furthermore, each $\calM_t'$ has
its own $L^2$ metric $g_{L_2, t}$, all naturally identified with one another by the dilation action. We then assert
that $(\calM_t', \calF_t^* g_{L^2,t})$ converges smoothly on compact sets to $(\calM_\infty', g_{\sfl})$.   We do
not belabor this point by writing this out more carefully since it is not used here in any substantial way.
Nonetheless, this picture is conceptually interesting in that it identifies the space of limiting configurations with
a certain `blowdown at infinity' of $\mc M_1$. We shall return to a closer examination of this phenomenon in another paper.

Let us now proceed with an alternate description of $\mc M_\infty'$. We may recast Definition~\ref{limconfdef1} into one involving harmonic metrics.
\begin{definition}\label{limconfdef2}
Let $(E,\bar\partial_E,\Phi)$ be a Higgs bundle such that $q= \det \Phi$ has only simple zeroes. A {\em limiting metric} is a flat hermitian metric $H_\infty$ on $E$ over $X_q^\times = X \setminus q^{-1}(0)$ such that $\Phi$ is normal with respect to $H_\infty$, i.e.\ the limiting equation
\[
F_{H_\infty} =0, \quad [\Phi \wedge \Phi^{\ast_{H_\infty}}]=0
\]	
is satisfied over $X_q^\times$. Here $F_{H_\infty}$ is the curvature of the Chern connection $A_{H_\infty}$ of $H_\infty$.
\end{definition}

Fixing a hermitian metric $H$, a limiting configuration is obtained from a limiting metric as follows. Express $H_\infty$ with respect to $H$ with an 
$H$-selfadjoint endomorphism field $\Xi_\infty$, so $H_\infty(\sigma, \tau)= H(\sigma, \Xi_\infty \tau)$ for any two sections $\sigma, \tau$ of $E$. 
Setting $\Xi^{-1}_\infty=g_\infty g^*_\infty$, then $H=g_\infty^*H_\infty$ and thus $A_\infty=A^{g_\infty}$ and $\Phi^\infty=g_\infty^{-1} \Phi g_\infty$ constitute 
a limiting configuration in the complex gauge orbit of the Higgs bundle $(A,\Phi)$.

The interpretation of the limiting metric for a Higgs bundle goes back to an observation by Hitchin and is described in detail 
in \cite{msww15}, see also   \cite {moch15}. We review this now. Fix $q \in H^0(K_X^2)$ with simple zeroes. As in \textsection \ref{subsect:algebrintsyst}, let $p_q: S_q \to X$ denote the spectral cover and $L_\pm \subset p_q^*E$ the eigenlines of $p_q^* \Phi$; 
these are exchanged by the involution $\sigma$. Then $L_+ = L \otimes p_q^* \Theta^*$ for the previously chosen square root $\Theta$ of the 
canonical bundle $K_X$ and a holomorphic line bundle $L \in \Prym(S_q)$, i.e.\ $\sigma^* L = L^*$. Then $L_-= \sigma^* L_+ = L^* 
\otimes p_q^* \Theta^*$. Since $q$ is holomorphic, $(q\bar q)^{1/4}$ is a flat hermitian metric on $\Theta^*$ over $X_q^\times$, hence 
on $p_q^*\Theta^*$ over $S_q^\times$, and is singular at the punctures. Furthermore, since $L$ is a holomorphic line bundle of zero degree, 
it admits a flat hermitian metric $h$. Altogether, we form the singular flat metric $h_+ = h (q\bar q)^{1/4}$ on $L_+$. If $A_h$ and $A_q$ 
denote the Chern connections of the metrics $h$ and $(q\bar q)^{1/4}$, respectively, then the Chern connection $A_{h_+}$ of $h_+$ is the 
tensor product of $A^h$ and $A^q$. Pulling back gives the metric $h_-=\sigma^*h_+$ on $L_-$, so that $h_+ \oplus h_-$ is $\sigma$-invariant 
on $L_+ \oplus L_-$ and thus descends to a limiting metric $H_\infty$ on $E$. (We use here that $p_q^* E$ decomposes holomorphically 
as the direct sum of the line bundles $L_+$ and $L_-$ on the punctured spectral curve $S_q^\times$.) 

Varying the holomorphic line bundle $L \in \Prym(S_q)$, we obtain all limiting configurations associated to $q$, which identifies
$\Prym(S_q)$ with the torus $\mc{M}_\infty(q)$ of limiting configurations associated to $q$, see Section 4.4 in \cite{msww14}. We describe 
this more concretely. Fix a $\mc{C}^\infty$-trivialization $\underline \C = S_q \times \C$ of the underlying line bundle, with standard 
hermitian metric $h_0$.  With respect to this metric, any holomorphic structure on this trivial bundle 
%$L \in \Prym(S_q)$ on $\underline \C$ $$ 
is represented by a flat unitary connection $d + \eta$, where $\eta \in \Omega^1(S_q, i \R)$ is closed and odd under the involution, 
$\sigma^* \eta = - \eta$.  Clearly, $d+\eta$ is the Chern connection of $h_0$ for the holomorphic structure $\bar\partial + \eta^{0,1}$ 
and $h_+=h_0(q\bar q)^{1/4}$ gives rise to the limiting metric $H_\infty$. The Chern connections satisfy $A_{h_+} = A_q + \eta$ and 
$A_{h_-} = A_q - \eta$ on $L_+$ and $L_-$, respectively.

There is also a Hitchin section in $\mc M_\infty$ corresponding to any choice of square root $\Theta = K_X^{1/2}$. Thus consider 
$E=\Theta \oplus \Theta^*$ with Higgs field
\[
\Phi= \begin{pmatrix} 0 & -q \\ 1 & 0 \end{pmatrix}.
\]
This has spectral data $L= \mathcal{O}_{S_q} \in \Prym(S_q)$, corresponding to $\eta=0$. Indeed, note that from \cite[Remark 3.7]{bnr87},
$E=(p_q)_*M$ for $M = L_+ \otimes p_q^* K_X$. However, $(p_q)_* \mc{O}_{S_q} = \mc{O}_X \oplus K_X^{-1}$, so by the push-pull formula, 
\[
(p_q)_*( p_q^* \Theta) = (p_q)_*( \mc{O}_{S_q} \otimes p_q^* \Theta) = (p_q)_*  \mc{O}_{S_q} \otimes \Theta = \Theta \oplus \Theta^*,
\]
and hence by the spectral correspondence, $M=p_q^* \Theta$. This shows that $L_+= p_q^* \Theta^*$ and so $L= \mc{O}_{S_q}$ as claimed. 
Let $H_\infty$ be the limiting metric for this Higgs bundle. 

\begin{lemma}
The limiting metric on the Higgs bundle $(E,\Phi)$ above is given up to scale by
\[
H_\infty = (q\bar q)^{-1/4} \oplus (q\bar q)^{1/4}
\] 
with respect to the decomposition $E=\Theta \oplus \Theta^*$.
\end{lemma}
\begin{proof}
It suffices to check that $\Phi$ is normal with respect to $H_\infty$ on the punctured surface $X^\times$. To that end, trivialize $\Theta^{\pm 1}$ locally 
by $dz^{\pm 1/2}$, so if $q = f dz^2$ then 
\[
H_\infty = \begin{pmatrix}
 |f|^{-1/2} & 0\\ 0 &  |f|^{1/2} \end{pmatrix} \quad \mbox{and} \qquad \Phi=\begin{pmatrix} 0 & f \\ 1 & 0 \end{pmatrix} dz.
\]
The eigenvectors $s_\pm= \pm \sqrt f \, dz^{1/2} + dz^{-1/2}$ satisfy $H_\infty(s_+,s_+) = H_\infty(s_-,s_-) = 2|f|^{1/2}$ and $H_\infty(s_+,s_-)=0$ on $X^\times$ as desired.
\end{proof}
 
As before, we consider the complex vector bundle $E$ with   background hermitian metric $H=k \oplus k^{-1}$ and  Chern connection $A_H=A_k \oplus A_{k^{-1}}$, and 
consider the limiting configuration $(A_\infty(q),\Phi_\infty(q))$ corresponding to $H_\infty$. In the following we write 
$|q|_k^{1/2}=(q\bar q)^{1/4}k$ where $| \cdot |_k$ is the norm on $K_X^2$ induced by $k$.

\begin{lemma}\label{lem:formulaHitchinlimiting}
The limiting configuration corresponding to the limiting metric $H_\infty = (q\bar q)^{-1/4} \oplus (q\bar q)^{1/4}$ is given by
\[
A_\infty(q)=A_H + \frac{1}{2} \left(\Im \bar \partial \log |q|_k\right)  \, 
\begin{pmatrix}
 i & 0 \\ 0 & -i
 \end{pmatrix}
\]
and
\[
\Phi_\infty(q)= \begin{pmatrix}
 0 & |q|_k^{-1/2} q \\ |q|_k^{1/2} & 0	
 \end{pmatrix}.
\]
with respect to the decomposition $E=\Theta \oplus \Theta^*$.
\end{lemma}

\begin{remark}
Note that if $z$ is a local holomorphic coordinate around a zero of $q$ such that $q=-zdz^2$ and $k$ is the flat metric induced by the holomorphic trivialization, these formul{\ae} reduce to the standard expression for the singular model solution 
\[
A_\infty^\fid=\frac 18\begin{pmatrix}1&0\\0&-1\end{pmatrix}\left(\frac{dz}{z}-\frac{d\bar z}{\bar z}\right),\quad\Phi_\infty^\fid=\begin{pmatrix}0&\sqrt{|z|}\\\frac{z}{\sqrt{|z|}}&0\end{pmatrix}dz, 
\]
considered in  \cite{msww14} and called there the limiting fiducial solution.
\end{remark}

%\begin{rem*}
%	Note that $(q\bar q)^{1/8}k^{1/2}$ is a positive function on $X^\times$: If $q=f dz^2$ and $k = \kappa \, dz^{-1/2} \otimes d \bar z^{-1/2}$, then $(q\bar q)^{1/8}k^{1/2} = |f|^{1/4} \kappa^{1/2}$. 
%\end{rem*}

\begin{proof}
Write $H_\infty(\sigma,\tau)= H(\sigma,\Xi_\infty \tau)$ where $\Xi_\infty$ is the $H$-selfadjoint endomorphism field
\[
\Xi_\infty = \begin{pmatrix} (q\bar q)^{-1/4}k^{-1} & 0 \\ 0 & (q\bar q)^{1/4}k \end{pmatrix}.
\]
If we then set
\[
g_\infty = \begin{pmatrix} (q\bar q)^{1/8}k^{1/2} & 0 \\ 0 & (q\bar q)^{-1/8}k^{-1/2} \end{pmatrix}
\]
then $H_\infty^{-1}=g_\infty g_\infty^*$.  This gives
\begin{align*}
g_\infty^{-1}( \bar\partial g_\infty) & =  \bar \partial \log \bigl((q\bar q)^{1/8}k^{1/2}\bigr) \begin{pmatrix}1 & 0\\
0 & -1 \end{pmatrix} %\\	
%& \overset{loc}{=} \Bigl(\frac{1}{8} \frac{\bar \partial \bar f}{\bar f} + \frac{1}{2} \frac{\bar\partial \kappa}{\kappa}\Bigr)\begin{pmatrix} 1 & 0 \\ 
% 0 & -1\end{pmatrix}
\end{align*}
and consequently
\begin{align*}
A_\infty &= A_H + g_\infty^{-1}\bar\partial g_\infty - (g_\infty^{-1}\bar\partial g_\infty)^*\\
&= A_H + 2 \Im \bar\partial\log\bigl((q\bar q)^{1/8}k^{1/2}\bigr) 
\begin{pmatrix}
 i & 0 \\ 0 & -i
\end{pmatrix}%\\
 %& \overset{loc}{=} A_H + \Bigl( \frac{1}{4} \Im \frac{\bar\partial |f|^2}{|f|^2} + \Im \frac{\bar \partial \kappa}{\kappa} \Bigr)\begin{pmatrix}
% i & 0 \\ 0 & -i
% \end{pmatrix} 
\end{align*}
and 
\begin{align*}
\Phi_\infty = g_\infty^{-1} \Phi g_\infty &=  \begin{pmatrix}
 0 & (q \bar q)^{-1/4}k^{-1} q \\ (q \bar q)^{1/4} k & 0
 \end{pmatrix}, 
\end{align*}
as desired. 
\end{proof}

Pulled back to the spectral curve, the limiting configuration attains the form
\[
p_q^* A_\infty(q) = (A_q \oplus A_q)^{g_\infty}, \quad \Phi_\infty(q)= g_\infty^{-1}\Phi g_\infty.
\]
More generally, if $(A_\infty(q,\eta),\Phi_\infty(q,\eta))$ denotes the limiting configuration corresponding to an element $L \in \Prym(S_q)$ 
determined by an odd 1-form $\eta \in \Omega^1(S_q;i\R)$ then
\[
p_q^* A_\infty(q,\eta)  = p_q^* A_\infty(q) + \eta \otimes g_\infty^{-1}\begin{pmatrix} 1 & 0 \\ 0 & -1 \end{pmatrix}g_\infty, \quad 
\Phi_\infty(q,\eta)=  \Phi_\infty(q). 
\]
Observe now that the pull-back bundle $p_q^* L_{\Phi_\infty}$ is spanned by the section $i s_\infty$ where 
\[
s_ \infty = g_\infty^{-1}\begin{pmatrix} 1 & 0 \\ 0 & -1 \end{pmatrix}g_\infty \in \Gamma(S_q^\times; p_q^* \End_0(E)).
\]
This section $s_\infty$ is parallel with respect to $A_\infty(q)$, so $p_q^* L_{\Phi_\infty}$ is trivial as a flat line bundle, i.e., isomorphic to 
$i \underline \R = S_q^\times \times i \R$ with the trivial connection.  Pulling back to $S_q^\times$, any section of $L_{\Phi_\infty}$ can be written 
as $f \cdot s_\infty$, where $f \in \mc{C}^\infty(S_q^\times,i \R)$ is odd with respect to the involution $\sigma$. Similarly, a 1-form with 
values in $L_{\Phi_\infty}$ corresponds via pull-back to $S_q^\times$ to an odd 1-form $\eta \in \Omega^1(S_q^\times, i\R)$, i.e. 
$\sigma^* \eta = - \eta$, so that $H^1(S_q^\times;i\R)_\odd = H^1(X^\times;L_{\Phi_\infty})$.  Under these identifications, 
\[
A_\infty(q,\eta)  = A_\infty(q) + \eta, \quad \Phi_\infty(q,\eta)=  \Phi_\infty(q).
\]
Define $H^1_\Z(S_q;i\R)_\odd \subset H^1(S_q; i \R)_\odd$ as the lattice of classes with periods in $2\pi i \,\Z$ and similarly the lattices 
$H^1_\Z(S_q^\times;i\R)_\odd \subset H^1(S_q^\times; i \R)_\odd$ and $H^1_\Z(X^\times;L_{\Phi_\infty}) \subset H^1(X^\times;L_{\Phi_\infty})$, 
cf.\ \cite[\S 4.4]{msww14}.
\begin{proposition}\label{corresp}
The map $d+\eta \mapsto A_\infty(q) + \eta$ induces a diffeomorphism
\[
\Prym(S_q) = \frac{H^1(S_q; i \R)_\odd}{H^1_\Z(S_q;i\R)_\odd} \overset{\cong}{\longrightarrow}  \frac{H^1(X^\times;L_{\Phi_\infty})}{H^1_\Z(X^\times;L_{\Phi_\infty})}=\mathcal{M}_\infty(q).
\]
\end{proposition}

%In the following we will drop the distinction between $\hat\eta$ and $\check\eta$ and simply refer to a form $\eta$  

In order to prove this proposition we need the following

\begin{lemma}\label{res_isom}
The restriction map
\[
H^1(S_q;i\R)_\odd \rightarrow H^1(S_q^\times;i\R)_\odd = H^1(X^\times;L_{\Phi_\infty})
\]	
is an isomorphism.  
\end{lemma}

\begin{proof}
In the following, imaginary coefficients are understood. Since $S_q^\times$ is a $\sigma$-invariant subset of $S_q$, there is a long exact cohomology sequence:
\[
\ldots \to H^p(S_q,S_q^\times)_\odd \to H^p(S_q)_{\odd} \to H^p(S_q^\times)_\odd \to  H^{p+1}(S_q,S_q^\times)_\odd \to\ldots
\]
By excision $H^p(S_q,S_q^\times) \cong \bigoplus_{i=1}^k H^p(D_i,D_i^\times)$ where $(D_i,D_i^\times) \cong (D,D^\times)$ are disks around the punctures $p_1, \ldots, p_k$ where $k=4\gamma-4$. Using the long exact sequence for the pair $(D,D^\times)$ together with the observation that $H^0(D^\times)_\odd=0$ (constants are even) and $H^1(D^\times)_\odd \cong H^1(S^1)_\odd=0$ (the angular form $d\theta$ is even) we obtain that $H^1(D,D^\times)_\odd= H^2(D,D^\times)_\odd=0$. It follows that the map $H^1(S_q)_\odd \to H^1(S_q^\times)_\odd$ is an isomorphism.
\end{proof}

For later use we record 

\begin{corollary}\label{L2_repr}
The restriction of the unique harmonic representative of a class in $H^1(S_q;i\R)_\odd$ yields a distinguished closed and coclosed representative of the corresponding class in $H^1(X^\times;L_{\Phi_\infty})$. This representative lies in $L^2$, i.e.\ is an $L^2$-harmonic 1-form.
\end{corollary}

\begin{proof}
  Since the restriction of the canonical projection $\pi : S_q\to X^\times$ to $\pi^{-1}(X^\times)$ is a conformal map and the
  space of $L^2$-harmonic $1$-forms is conformally invariant in 2 dimensions, it follows that $L^2$-harmonic $1$-forms are preserved under    pull-back along $\pi$. 
\end{proof}

\begin{definition}\label{L2_harmonic_forms}
Let
\[
\mathcal{H}^1(X^\times;L_{\Phi_\infty})= \bigl\{ \eta \in \Omega^1(X^\times,L_{\Phi_\infty}): p_q^*\eta \in \mathcal{H}^1(S_q; i\R)_{\odd} \bigr\} 
\]	
be the corresponding space of $L^2$-harmonic forms on $X^{\times}$.
\end{definition}

\begin{proof}[Proof of Proposition \ref{corresp}]
It remains to check that the isomorphism from Lemma \ref{res_isom} is compatible with the integer lattices. This is clearly the case for the map $H^1(S_q;i\R)_\odd \rightarrow H^1(S_q^\times;i\R)_\odd$. Now $\eta \in \Omega^1(S_q^\times,i \R)_\odd$ represents a class in $H^1_\Z(S_q^\times;i\R)_\odd$ if and only if it is of the form $g = d \log g$ for $g \in \mc{C}^\infty(S_q^\times, S^1)_\odd$. Since $g$ corresponds to a unitary gauge transformation commuting with $\Phi_\infty$ on $X^{\times}$ this is equivalent to $\eta \in \Omega^1(X^\times; L_{\Phi_\infty})$ representing a class in $H^1_\Z(X^\times;L_{\Phi_\infty})$.
\end{proof}

 As a final remark here, we include the 
\begin{proposition}
  The family of lattices $H^1_{\mathbb Z} (S_q; i\RR)_{\mathrm{odd}} \cong H^1_{\mathbb Z} (X^\times; L_{\Phi_\infty})$ over $\calB'$ are
  naturally identified with the local system $\Gamma$ which is defined using the algebraic completely integrable system structure, cf.\ Proposition~\ref{firstdiffeo}.  Therefore, as noted in the introduction,
  there is a natural diffeomorphism between the quotients
  \[
\mathcal A = T^* \calB'/\Gamma \cong  M_\infty'
\]
which intertwines the $\CC^\times$ action on both sides.
  \label{seconddiffeo}
  \end{proposition}

%%%%%%%%%%%%%%%%%%%%%%%%%%%%%%%%%%%%%%%%%%%%%%%%%
\subsection{Horizontal directions}
%%%%%%%%%%%%%%%%%%%%%%%%%%%%%%%%%%%%%%%%%%%%%%%%%
 Recall that   the Gau{\ss}-Manin connection on the Hitchin fibration gives rise to a splitting of each tangent space of $\mc M'$ into a direct sum of vertical and horizontal subspaces. This is the sense in which the terms horizontal and vertical are used in the following. The remainder of this section is devoted to deriving useful expressions for the metric applied to horizontal, vertical, and mixed pairs of tangent vectors.

The Hitchin section is a horizontal Lagrangian submanifold in $\mc M'$  as follows from the local symplectomorphism between 
$(T^*\mc B',\omega_{T^*\mc B'})$ and $(\mc M',\eta)$, cf.\ \textsection \ref{subsect:algebrintsyst}.
%cf.\ the remark in Section~\ref{spe.kah.met}. 
 Any smooth family of 
holomorphic quadratic differentials $q(s)\in\mc B'$ can thus be lifted to a family of Higgs bundles $\mc H(s)=(E,\Phi(s))$ in the Hitchin section. 
Fixing a hermitian metric $H$ on $E$, we denote the family of limiting configurations corresponding to $(A_H,\Phi(s))$ by $(A_\infty(s), 
\Phi_\infty(s))$. Setting $q:=q(0)$ and $\dot q:= \left.\frac{\partial}{\partial s}\right|_{s=0} q(s)$, then a brief calculation shows that
\[
\dot A_\infty := \left.\frac{\partial\, }{\partial s}\right|_{s=0}A_\infty(s)
= - \frac 14 d \Im (\dot q/q) \begin{pmatrix}
 i & 0 \\ 0 & -i
 \end{pmatrix}
\]  
and
\begin{align*}
\dot \Phi_\infty :=\left. \frac{\partial}{\partial s}\right|_{s=0}\Phi_\infty(s) 
&=\begin{pmatrix}
0 & |q|_k^{-1/2} \bigl(-\frac{1}{2} \Re(\dot q/q) q + \dot q\bigr)  \\ \tfrac{1}{2}|q|_k^{1/2}\Re(\dot q/q) & 0	
\end{pmatrix}.
\end{align*}

Assuming the zeroes of $\dot q$ do not coincide with those of $q$, or equivalently, the deformation is not radial, then $\dot A_\infty$ 
has double poles at the zeroes of $q$, so $\dot A_\infty \not\in L^2$.  However,  $\dot A_\infty$ is pure gauge 
and $(\dot A_\infty, \dot \Phi_\infty)$ can be transformed to lie in $L^2$, albeit with a singular gauge transformation. In addition, this 
gauged variation even satisfies the Coulomb gauge  condition \eqref{gaugefixing}, and its $L^2$ norm turns out to be simply the semiflat metric. 

To be more precise, set
\begin{equation}\label{inf_gauge}
\gamma_\infty := - \frac 14 \Im(\dot q/q) \begin{pmatrix}
 i & 0 \\ 0 & -i
 \end{pmatrix}. 
\end{equation}
Then 
\[
\alpha_\infty:=\dot A_\infty - d_{A_\infty} \gamma_\infty =0
\]
and 
\begin{align}
\label{eq:varphiinfty}\varphi_\infty:=\dot \Phi_\infty - [\Phi_\infty \wedge \gamma_\infty] &= \begin{pmatrix}
0 & \frac 12 |q|_k^{-1/2} \dot q  \\ \tfrac{1}{2} |q|_k^{1/2} \dot q/q & 0	
\end{pmatrix},
\end{align}
so clearly, $(\alpha_\infty, \varphi_\infty)=(0,\varphi_\infty)$ is in $L^2$. 

We next show that $(0,\varphi_\infty)$ satisfies the Coulomb gauge condition, again with the caveat that this is accomplished 
only by a singular gauge transformation.  
\begin{lemma}
The pair $(0, \varphi_\infty)$ satisfies $d_{A_\infty}^* \alpha_\infty -2  \pi^{\skew}(i\ast  [\Phi_\infty^* \wedge \varphi_\infty]) =0$. \end{lemma}
\begin{proof}
Since $\alpha_\infty=0$, it suffices to show that $[\Phi_\infty^* \wedge \varphi_\infty]=0$.  Using the local holomorphic frame $dz^{\pm 1/2}$ for
$E=\Theta \oplus \Theta^*$, 
\[
H = \begin{pmatrix} \kappa & 0 \\ 0 & \kappa^{-1} \end{pmatrix}
\]
and hence 
\[
\Phi_\infty = \begin{pmatrix} 0 & |f|^{-1/2} \kappa^{-1}f \\ |f|^{1/2} \kappa & 0 \end{pmatrix} dz.
\]
Now one easily calculates
\[
\Phi_\infty^* = \begin{pmatrix} 0 & |f|^{-1/2} \kappa^{-1} \\ |f|^{-1/2} \kappa \bar f& 0 \end{pmatrix} dz , \quad \varphi_\infty = \begin{pmatrix} 0 & \frac 12 |f|^{-1/2} \kappa^{-1} \dot f \\ \frac 12 |f|^{1/2} \kappa \dot f/f & 0 \end{pmatrix} dz
\]
and finally
\[
[\Phi_\infty^* \wedge \varphi_\infty ] = \frac 12 ( |f|\dot f/f - |f|^{-1} \bar f \dot f) \begin{pmatrix} 1 & 0 \\ 0 & -1 \end{pmatrix} d\bar z \wedge dz =0
\]
as claimed.
\end{proof}

Finally, the following result  follows directly from the definitions and formul\ae\ above.
\begin{proposition}
One has the identity
\[
g_{sK}(\dot q, \dot q) = \int_X |\varphi_\infty|^2\, dA, 
\]
where $\varphi_\infty$ is defined by \eqref{eq:varphiinfty}.
\end{proposition}

 We have now shown that the restriction of $g_\sfl$ and this renormalized $L^2$ metric (i.e.\ the $L^2$ metric obtained on $\mathcal M_{\infty}'$ by admitting singular gauge transformations to put tangent vectors into Coulomb gauge) are the same on tangent vectors 
to the Hitchin section on the space of limiting configurations. 

To make the analogous computations at limiting configurations which are not on the Hitchin section, we construct more 
general horizontal lifts of families $q(s)$ in $\mc B'$. Recall that if $q \in H^0(K_X^2)$ is fixed and $(A_\infty, \Phi_\infty)$ is any base point 
in $\pi^{-1}(q)$, then any element in this fiber takes the form
\begin{equation}\label{eq:genlimitingconf}
(A_\infty + \eta, \Phi_\infty)\ \quad \mbox{where}\ \ [\eta \wedge \Phi_{\infty}]=0 \ \mbox{and}\ d_{A_\infty}\eta = 0. 
\end{equation}
Write $A_\infty(s)$, $\Phi_\infty(s)$ and $\eta(s)$ for the horizontal lifts, and assume that $((A_\infty(0),\Phi_\infty(0))$ lies in the Hitchin section over $q$;
then differentiating the defining conditions $[\eta(s) \wedge \Phi_\infty(s)]=0$ and $d_{A_\infty(s)} \eta(s)=0$ gives 
\begin{equation}\label{eins}
[\dot \eta \wedge \Phi_\infty] + [ \eta \wedge \dot \Phi_\infty] = 0	
\end{equation}
and 
\begin{equation}\label{zwei}
d_{A_\infty} \dot \eta + [ \dot A_\infty \wedge \eta]=0	
\end{equation}
at $s=0$. These two equations characterize the tangent vectors $(\dot A_\infty+\dot \eta, \dot \Phi_\infty)$ to the space of limiting configurations 
$\mathcal{M}_\infty$ in $\pi^{-1}(q)$. 

We shall use $\gamma_\infty$, the infinitesimal gauge transformation which regularizes $A_\infty$, to generate all horizontal lifts of $\dot q$.
Note that since $d_{A_\infty} \gamma_\infty = \dot A_\infty$, we have
\[
d_{A_\infty+\eta} \gamma_\infty = d_{A_\infty} \gamma_\infty + [\eta \wedge \gamma_\infty] = \dot A_\infty + [\eta \wedge \gamma_\infty].
\]
\begin{lemma}\label{lem:tangdirectionslimconf}
Setting $\dot \eta = [\eta \wedge \gamma_\infty]$, then equations \eqref{eins} and \eqref{zwei} are satisfied, hence 
$(\dot A_\infty + \dot \eta, \dot \Phi_\infty)$ is the horizontal lift of $\dot q$ at $(A_\infty+\eta, \Phi_\infty)$.
\end{lemma}
\begin{proof}
By the Jacobi identity,
%\begin{align*}
\begin{multline*}
[\dot \eta \wedge \Phi_\infty] + [ \eta \wedge \dot \Phi_\infty] = [[\eta \wedge \gamma_\infty], \Phi_\infty] + [ \eta \wedge \dot \Phi_\infty]\\
= [\gamma_\infty \wedge [\Phi_\infty \wedge \eta]] - [\eta \wedge [\Phi_\infty \wedge \gamma_\infty]] + [\eta \wedge \dot \Phi_\infty] = 
[\gamma_\infty \wedge [\Phi_\infty \wedge \eta]] + [ \eta \wedge \varphi_\infty] = 0,
\end{multline*}
%\end{align*}
since $\varphi_\infty = \frac 12 \frac{\dot q}{q} \Phi_\infty$ and $[\eta \wedge \Phi_\infty]=0$.  Furthermore, 
\begin{align*}
d_{A_\infty} \dot \eta + [ \dot A_\infty \wedge \eta] &= d_{A_\infty} [ \eta \wedge \gamma_\infty] + [\dot A_\infty \wedge \eta ]\\
&= [d_{A_\infty} \eta \wedge \gamma_\infty ] - [ \eta \wedge d_{A_\infty}\gamma_\infty] + [\dot A_\infty \wedge \eta] = 0
\end{align*}
using $d_{A_\infty} \eta =0 $ and $d_{A_\infty} \gamma_\infty = \dot A_\infty$. By definition, $\dot A_\infty + \dot \eta = d_{A_\infty+\eta}\gamma_\infty$
is pure gauge, which means that $(\dot A_\infty + \dot \eta, \dot \Phi_\infty)$ is horizontal with respect to the Gau{\ss}-Manin connection.
\end{proof}

As before, applying $\gamma_\infty$ to $\dot \Phi_\infty$  gives the gauge equivalent infinitesimal deformation
$(0, \varphi_\infty)$ of $(A_\infty + \eta, \Phi_\infty)$.   The following is then an immediate consequence of the
fact that the Hitchin fibration is a Riemannian submersion. 
\begin{corollary}\label{cor:l2horequalssf} 
One has
\begin{equation*}
g_{\sfl}(\dot{q}^{\hor},\dot{q}^{\hor}) = \int_X |\varphi_\infty|^2\, dA
\end{equation*}
where $\dot{q}^{hor}$ denotes the horizontal lift of $\dot{q} \in H^0(K_X^2)$.
\end{corollary}
 
\subsection{Vertical directions}
Now fix $q\in H^0(K_X^2)$ and $(A_\infty,\Phi_\infty) \in \pi^{-1}(q)$. As we have remarked, up to gauge, any element in $\pi^{-1}(q)$ takes 
the form $(A_\infty+ \eta, \Phi_\infty)$ where $\eta \in \Omega^1(L_{\Phi_\infty})$ satisfies $d_{A_\infty}\eta=0$. The infinitesimal gauge 
action shifts $\eta$ by $d_{A_\infty} \gamma$, $\gamma \in \Omega^0(L_{\Phi_\infty})$. Hence the vertical tangent space is identified 
with the cohomology space
\[
H^1(L_{\Phi_\infty}) = \frac{\ker (d_{A_\infty}\colon \Omega^1(L_{\Phi_\infty}) \to \Omega^2(L_{\Phi_\infty}))}{\im (d_{A_\infty}\colon\Omega^0(L_{\Phi_\infty}) 
\to \Omega^1(L_{\Phi_\infty}))}.
\]

Each class in $H^1(X^\times;L_{\Phi_\infty})$ possesses a distinguished closed and coclosed $L^2$ representative $\alpha_\infty$.  By Lemma 
\ref{res_isom} and Corollary \ref{L2_repr}, $\alpha_\infty$ is the restriction of the unique harmonic representative of the corresponding class 
in $H^1(S_q;i \R)_\odd$. 
\begin{lemma}
If $(\dot A_\infty,\dot \Phi_\infty) = (\alpha_\infty, 0)$ where $\alpha_\infty \in \Omega^1(L_{\Phi_\infty})$  is the harmonic representative, 
then  
\[
d_{A_\infty}^*\dot A_\infty-2 \pi^{\skew}(i \ast[\Phi_\infty^* \wedge \dot \Phi_\infty])=0.
\]
\end{lemma}

\begin{proof}
This is a trivial consequence of $\alpha_\infty$ being coclosed and $\dot \Phi_\infty=0$.	
\end{proof}

\begin{proposition} \label{sfleta}
If $\alpha_\infty$ is as above then   
\[
g_{\sfl}(\alpha_\infty,\alpha_\infty) = \int_X |\alpha_\infty|^2dA.
\]	
\end{proposition}

\begin{proof}
This follows from the above discussion along with Equation \eqref{vertical_L2}.
\end{proof}

\subsection{Mixed terms}
\begin{lemma}
If $v^{\hor}=(\dot A_\infty,\dot \Phi_\infty)$ is the horizontal lift of $\dot q \in H^0(K_X^2)$ and $w^{\verti}=(\alpha_\infty,0)$ is 
a vertical tangent vector with $\eta$ harmonic, then 
\[
\langle v^{\hor}, w^{\verti} \rangle \equiv 0
\]
pointwise. Therefore, the $L^2$ inner product of these two vectors vanishes.  Hence the off-diagonal parts of the $L^2$ inner product and the semiflat inner product agree.
% \[
% \int_{X \setminus B_\varepsilon(\mathfrak{p})} \langle v^{\hor},v^{\verti} \rangle =0. 
% \]
\end{lemma}
\begin{proof}
The gauged tangent vector corresponding to a horizontal deformation $(\dot A_\infty, \dot \Phi_\infty)$ 
is of the form $(0, \varphi_\infty)$, while the gauged tangent vector corresponding to a vertical deformation
is of the form $(\alpha_\infty, 0)$.  These are clearly orthogonal pointwise. On the other hand, the orthogonality
of vertical and horizontal tangent vectors in the semiflat metric is part of the definition.
\end{proof}

%\medskip
%We conclude this section by remarking that if one uses only the formul\ae\ in this section, it is not easy to conclude that 
%$g_{\sf}$ depends smoothly on the points in $\calM'$. Indeed, taking higher derivatives $|\varphi_\infty|^2$ in the base direction 
%seems to lead to nonintegrable singularities at the zeroes of $q$. However, using the identifications in this section and 
%appealing back to Corollaries~\ref{smoothsK} and \ref{smoothsf}, we see that the renormalized $L^2$ metric on
%limiting configurations is indeed smooth. 

\section{The approximate moduli space}\label{approximate}
Our goal is to understand the asymptotics of the $L^2$ metric on the open subset $\mc M'$ of the Hitchin moduli space.
In this section we recall and slightly recast the construction of approximate solutions from \cite{msww14} in terms of 
parametrized families of data and solutions, and then use these families to define and study the $L^2$ metric on $\mc M'$.

In more detail, consider a smooth slice $\calS_\infty$ in the `premoduli space' $\mc P\!\mc M_\infty'$, which consists of
the solutions to the uncoupled Hitchin equations {\em before passing to the quotient} by unitary gauge transformations.
The slice $\calS_\infty$ gives a coordinate chart on $\mc M_\infty'$. The construction in \cite{msww14} produces
from the elements in $\calS_\infty$ a smooth family of approximate solutions $\calS^\app$ of the self-duality equations, 
and then perturbs each element of $\calS^\app$ to an exact solution.  We add to this, cf.\ the discussion in \S 10,
the observation that this final perturbation map is smooth in these parameters, so we obtain a slice $\calS$ in the space of 
solutions to the Hitchin equations, which in turn corresponds to a coordinate chart in $\mc M'$.  

In the previous section we studied the $L^2$ inner products of renormalized gauged tangent vectors on $\mc P\!\mc M_\infty'$
and showed that these correspond precisely to the inner products for the semiflat metric.  The construction above
yields tangent vectors, initially to the slice $\calS^\app$, and then to the slice $\calS$.  To analyze the $L^2$ 
metric we first put these tangent vectors into Coulomb gauge and then compute the appropriate integrals defining the metric. 
Each of these steps introduces correction terms to $g_{\sfl}$. The next four sections contain details of this
for pairs of tangent vectors to the approximate moduli space which are, respectively, horizontal, radial, vertical 
and `mixed'.  The main correction terms arise here. The final \S 10 shows that only an exponentially small
further correction is introduced when passing from the approximate to the true moduli space. 

The construction of an approximate solution is based on a gluing construction. In the initial step, a limiting configuration 
$S_\infty = (A_\infty, \Phi_\infty)$ is modified in a neighborhood of each zero of $q = \det \Phi_\infty$ by replacing 
it there with a desingularizing `fiducial' solution $(A_t^\fid, \Phi_t^\fid)$. This yields a pair $S_t^{\appr} = 
(A^\app_t, \Phi^\app_t)$ which is an approximate solution for the Hitchin equations in the sense that  $\mu( S_t^{\appr}) = \calO(e^{-\beta t})$ for some $\beta > 0$. 
It is straightforward to check that this construction may be done smoothly in all parameters.  Thus from
a smooth finite dimensional family $\calS_\infty$ of limiting configurations transverse to the gauge orbits, we 
obtain a smooth finite dimensional family of fields $\calS^\app$.  We think of this family as a submanifold
of a premoduli space $(\mc{PM}^\app)'$ of approximate solutions, which hence determines 
a coordinate chart in the approximate moduli space $(\mc M^\app)'$.  Since this discussion is local in the moduli 
spaces, we may work entirely with these slices, and so do not need to define this approximate moduli
space carefully. For convenience, however, we shall frequently refer to tangent vectors to $(\mc M^\app)'$,
which are tangent vectors to $\calS^\app$ which have been further modified to satisfy the gauge condition.
All of this is done, of course, only in some fixed neighborhood of infinity in the Hitchin base, $\mc B' \cap 
\{q: \|q\|_{L^1} \geq t_0^2\}$. 

To be more specific, fix $q \in \mc B'$ and let $(A_{\infty},\Phi_{\infty})$ denote the unique limiting configuration for the Hitchin section 
with  $\det \Phi_{\infty}=q$. By \eqref{eq:genlimitingconf}, a general limiting configuration takes the form  $(A_{\infty}+\eta,\Phi_{\infty})$ 
where $\eta$ is a suitable $d_{A_{\infty}}$-closed $1$-form commuting with $\Phi_{\infty}$. The connection $A_{\infty}$ is flat and has 
nontrivial monodromy around each zero of $q$, hence $H^1(\D^\times, d_{A_\infty}) = 0$, cf.\   \cite[Eq.\ (32)]{msww14}.  
Thus $\eta = d_{A_\infty} \gamma$ on each such punctured disk. As follows from \cite[Prop.\ 4.7]{msww14}, $|\gamma| = \calO(r^{1/2})$. 
Therefore we may modify $A_{\infty}+\eta$ by an exact $L_{\Phi_{\infty}}$-valued $1$-form so as to assume that $\eta\equiv0$ on 
$\bigsqcup_{p\in\mathfrak p}\D_p$.

Following \cite[\textsection 3.2]{msww14}, we define the family of desingularizations 
$S_t^{\appr}:=(A_t^{\appr}+\eta,t\Phi_t^{\appr})$ by  % of the limiting configuration $(A_{\infty}+\eta,\Phi_{\infty})$, by setting 
%\begin{equation}
\begin{eqnarray}
\label{eq:atappr}
A_t^{\appr} &=& A_{H} + \bigl(\tfrac{1}{2}+\chi(|q|_k)(4 f_t(|q|_k)-\tfrac{1}{2})\bigr)\Im \bar \partial \log |q|_k \begin{pmatrix}  i & 0 \\ 0 & -i \end{pmatrix} \\[0.5ex]
%\begin{equation}
\label{eq:phitappr} 
\Phi_t^{\appr} &=& \begin{pmatrix}  0 & |q|_k^{-1/2} e^{-\chi(|q|_k)h_t(|q|_k)}q\\|q|_k^{1/2}e^{\chi(|q|_k)h_t(|q|_k)} & 0 
 \end{pmatrix}. 
\end{eqnarray}
% \end{equation}
Here $h_t(r)$ is the unique solution to $(r\partial_r)^2 h_t = 8  t^2 r^3 \sinh 2h_t$ on $\RR^+$ with specific asymptotic 
properties at $0$ and $\infty$, and $f_t:=\frac{1}{8}+\frac{1}{4}r\del_r h_t$. Further, $\chi: \R^+ \to [0,1]$ is a suitable cutoff-function. The parameter $t$ can be removed from the equation for $h_t$ by substituting $\rho = \frac{8}{3} t r^{3/2}$; thus
if we set $h_t(r) = \psi(\rho)$ and note that $r\partial_r = \frac{3}{2} \rho \partial_\rho$, then
\[
(\rho \partial_\rho)^2 \psi = \frac{1}{2} \rho^2 \sinh 2\psi.
\]
This is a Painlev\'e III equation; there exists a unique solution which decays exponentially as $\rho \to \infty$ and with 
asymptotics as $\rho \to 0$ ensuring that $A_t^{\appr}$ and $\Phi_t^{\appr}$ are regular at $r=0$. More specifically, 
%\begin{equation}
\[
\begin{array}{rl}
\bullet\ & \psi(\rho) \sim -\log (\rho^{1/3} \left( \sum_{j=0}^\infty a_j \rho^{4j/3}\right), \quad \rho \searrow 0; \\[0.5ex]
\bullet\ & \psi(\rho) \sim K_0(\rho) \sim \rho^{-1/2} e^{-\rho}\sum_{j=0}^\infty b_j \rho^{-j}, \quad \rho \nearrow \infty; \\[0.5ex]
\bullet\ & \psi(\rho)\mbox{ is monotonically decreasing (and strictly positive) for $\rho > 0$}.
\end{array}
%\label{proph}
%\end{equation}
\]
These are asymptotic expansions in the classical sense, i.e., the difference between the function and the first $N$ terms
decays like the next term in the series, and there are corresponding expansions for each derivative. 
The function $K_0(\rho)$ is the Bessel function of imaginary argument of order $0$. 
\newpage
In the following result and for the rest of the paper, any constant $C$ which appears in an estimate is assumed to be independent of $t$.

\begin{lemma}\cite[Lemma 3.4]{msww14}\label{f_t-h_t-function}
The functions $f_t(r)$ and $h_t(r)$ have the following properties: 
\begin{enumerate}[(i)]
\item As a function of  $r$, $f_t$ has a double zero at $r=0$ and increases monotonically from $f_t(0) = 0$ to the limiting value $1/8$ as 
$r \nearrow \infty$.  In particular, $0 \leq f_t \leq \frac 18$.
\item As a function of $t$, $f_{t}$ is also monotone increasing. Further, $\lim_{t \nearrow \infty} f_t = f_\infty \equiv \frac18$ 
uniformly in $\mc C^\infty$ on any half-line $[r_0,\infty)$, for $r_0 > 0$. 
\item There are estimates 
\[
\sup_{r >0}  r^{-1} f_t(r) \leq C t^{2/3} \quad \text{and}\quad \sup_{r >0} r^{-2} f_t(r) \leq C t^{4/3}.
\]
\item When $t$ is fixed and $r \searrow 0$, then $h_t(r) \sim -\tfrac{1}{2} \log r + b_0 + \ldots$, where $b_0$ is an explicit constant. 
On the other hand, $|h_t(r)| \leq C \exp( -\tfrac83 t r^{3/2})/ ( t r^{3/2})^{1/2}$  for $t \geq t_0 > 0$, $r \geq r_0 > 0$.    
\item Finally,
\[
\sup_{r \in(0,1)} r^{1/2} e^{\pm h_t(r)} \leq C, \quad t \geq 1. 
\]
\end{enumerate}
\end{lemma}
It follows from the results in \cite{msww14} that the approximate solution  $S_t^{\appr}$  satisfies the self-duality equations up to an 
exponentially decaying error as $t\to\infty$, and   there is an exact solution $(A_t,\Phi_t)$ in its complex gauge orbit (unique
up to real gauge transformations) which is no further than $Ce^{-\beta t}$ pointwise away for some $\beta > 0$. 

%The set of all  approximate solutions obtained in this way constitutes the smooth Banach manifold $(\mathcal M^{\app})'$.

\section{Gauge correction}\label{subsect:gaugecorrect}
The $L^2$ metric is defined in terms of infinitesimal deformations which are orthogonal to the gauge group action. An arbitrary
tangent vector can be brought into this form by solving the gauge-fixing equation on all of $X$.
We first describe gauge-fixing in general and then estimate the gauge correction term in this particular instance. 

At the end of \S 2.4.2, we introduced the deformation complex and its differentials $D^1_{(A,\Phi)}$ and $D^2_{(A,\Phi)}$,
as well as the condition \eqref{gaugefixing} for an infinitesimal deformation $(\dot A, \dot \Phi)$ to be in gauge. 
% Section \ref{subsect:hkmetrics} the deformation complex 
% \begin{multline*}
% 0 \to \Omega^0(\su(E)) \xrightarrow{D^1_{(A,\Phi)}}  \Omega^1(\su(E)) \oplus \Omega^{1,0}(\mf{sl}(E))  \\
% \xrightarrow{D^2_{(A,\Phi)}} \Omega^2(\su(E)) \oplus \Omega^{1,1}(\mf{sl}(E))\to 0
% \end{multline*}
% at a solution $(A,\Phi)$ to the Hitchin equations. The first differential is the infinitesimal gauge action
% \[
% D^1_{(A,\Phi)} (\gamma) = (d_A \gamma, [\Phi \wedge \gamma]),
% \]
% while $D^2_{(A,\Phi)}$ is the linearization of the Hitchin map. The tangent space of the moduli space at $(A,\Phi)$ is identified with 
% the quotient 
% \[
% \ker D^2_{(A,\Phi)} / \im D^1_{(A,\Phi)} \cong \ker D^2_{(A,\Phi)}\cap(\im D^1_{(A,\Phi)})^\perp.
% \] 
% Since
% \[
% \int_X \langle d_A \gamma, \dot A \rangle = \int_X \langle \gamma, d_A^* \dot A \rangle, \ \ 
% \mbox{and}\quad 
% \int_X \langle [\Phi \wedge \gamma], \dot \Phi \rangle =  \int_X \langle \gamma, \ast \Re [ \Phi^* \wedge \dot \Phi ]\rangle,
% \]
% we see that 
% \begin{align*}
% (\dot A, \dot \Phi) \perp \im D^1_{(A,\Phi)} \quad \Longleftrightarrow \quad d_A^* \dot A + \ast \Re [ \Phi^* \wedge \dot \Phi ] =0.
% \end{align*}
% In this case we say that $(\dot A, \dot \Phi)$ is in gauge.  

\begin{lemma}[Infinitesimal gauge fixing] If $(\dot A, \dot \Phi)$ is an infinitesimal deformation of a solution $(A,\Phi)$ to the 
Hitchin equations, then there exists a unique  $\xi \in \Omega^0(\su(E))$ such that 
$(\dot A, \dot \Phi)-D^1_{(A,\Phi)}\xi$ is in gauge.  %=(\dot A - d_A \xi, \dot \Phi - [\Phi \wedge \xi])$ 
The same is true if $(A,\Phi)$ is sufficiently close to a solution to the Hitchin equations.
\end{lemma}
\begin{proof}
First suppose that $\mu(A,\Phi) = 0$. The transformed pair $(\dot A - d_A \xi, \dot \Phi - [\Phi \wedge \xi])$ is in gauge 
if and only if
\[
(D^1_{(A,\Phi)})^*  ( (\dot A, \dot \Phi) - D^1_{(A,\Phi)}\xi) = 0
%d_A^*(\dot A - d_A \xi) + \Re [ \Phi^* \wedge (\dot \Phi - [\Phi \wedge \xi])] =0
\]
or equivalently,
\begin{equation}
\calL_{(A,\Phi)} \xi = d_A^*\dot A -2 \pi^{\skew}(i \ast [\Phi^* \wedge \dot \Phi]),
\label{eq:gaugefixeq}
\end{equation}
where 
\begin{equation}
\calL_{(A,\Phi)} :=  (D^1_{(A,\Phi)})^* D^1_{(A,\Phi)} = \Delta_A - 2\pi^{\skew}(i \ast[ \Phi^* \wedge [ \Phi \wedge \cdot ] ]). 
\label{calLdef}
\end{equation}This operator already played a role in \cite{msww14}, albeit acting on $i\mathfrak{su}(E)$ rather than $\mathfrak {su}(E)$. 
Now
\[
\langle \calL \xi, \xi\rangle = \| d_A \xi\|^2 + 2 \| \, [ \Phi \wedge \xi] \, \|^2,
\]
so solutions to $\calL \xi = 0$ are parallel and commute with $\Phi$. But as already used in \cite{msww14},  
if $q = \det \Phi$ is simple, then the solution $(A,\Phi)$ must be irreducible.  This implies that $\calL$ is bijective, 
and so \eqref{eq:gaugefixeq} admits a unique solution. 

If $(A,\Phi)$ is sufficiently close to an exact solution, then $\calL_{(A,\Phi)}$ remains invertible and hence the conclusion is true then as well. 
\end{proof}

For an approximate solution $S_t^{\appr} = (A_t^{\appr}, t \Phi_t^{\appr})$,  define 
$$
 M_t\xi:=M_{\Phi_t^{\appr}}\xi:=-2 \pi^{\skew}(i\ast [(\Phi_t^{\appr})^{\ast} \wedge [\Phi_t^{\appr} \wedge \xi]])
$$
and also set
\begin{align*}
&D^1_t\xi:=D^1_{(A_t^{\appr}+\eta,t\Phi_t^{\appr})}\xi=(d_{A_t^{\appr}}\xi+[\eta \wedge\xi],t[\Phi_t^{\appr},\xi]),\\
& \calL_t\xi:=(D^1_t)^{\ast}D^1_t\xi=\Delta_{A_t^{\appr}+\eta} \xi -2t^2 \pi^{\skew}(i\ast [(\Phi_t^{\appr})^{\ast} \wedge [\Phi_t^{\appr} \wedge \xi]]). 
\end{align*}
Note that for any pair $(A_t, t \Phi_t)$, 
\[
\calL_t = \Delta_{A_t} + t^2 M_t.
  \]

\subsection{\bf Analysis of $\calL_t^{-1}$} 
We now study the inverse $G_t = \calL_t^{-1}$, recalling from \cite[Proposition 5.2]{msww14} that $\calL_t$ is uniformly invertible
when $t$ is large, 
\begin{equation}
\|G_t f\|_{L^2(X)} \leq C \| f\|_{L^2(X)},
\label{l2bdgt}
\end{equation}
where $C$ does not depend on $t$. This estimate controls the size of the gauge-fixing terms below. However, we 
require finer information about these terms, so we now examine the structure and mapping properties of this inverse
more closely. 

By construction, the approximate solution $(A_t^{\appr}, t \Phi_t^{\appr})$ is precisely equal to a fiducial solution inside each $\D_p$. 
This simplifies the results and arguments below, though these all have analogues if this is not the case, e.g.\ when $(A,t \Phi)$ is 
an exact solution. 

We first examine the scaling properties of the operator $\calL_t$ in each $\D_p$.
{Set $\varrho = t^{2/3} r$ (note the difference with
the previous change of variables $\rho = \frac83 t r^{3/2}$ used earlier). The coefficients of $A_t$ depend only on $\varrho$,  
and the $d\theta$ in $A_t$ does not need to be transformed. 
Write $\Delta_{A_t} = r^{-2} \widehat{\Delta_t}$, where $\widehat{\Delta_t}  = -(r\del_r)^2 + (-i \del_\theta + a(t^{2/3} r))^2$ for
some hermitian matrix $a$. Now $r\del_r = \varrho\del_\varrho$, so $\widehat{\Delta_t}$ can be reexpressed (in $\D_p$) as an operator 
$\widehat{\Delta_\rho}$ which depends   on $(\varrho, \theta)$ but not on $t$. The prefactor $r^{-2}$ equals 
$t^{4/3} \varrho^{-2}$, so 
\[
\Delta_{A_t} = t^{4/3} \varrho^{-2} \widehat{\Delta_\varrho} := t^{4/3} \Delta_\varrho.
\]

The second  term $t^2 M_t$ appearing in $\calL_t$ behaves similarly. Indeed, the matrix entries of $\Phi_t$ and $\Phi_t^*$ equal $r^{1/2}$ times functions of 
$t^{2/3} r = \varrho$, so that
\[
t^2 M_t = t^2 r \widehat{M}_\rho :=  t^{4/3} M_{\varrho},
\] 
where $M_\varrho = \rho \widehat{M}_\varrho$ is an endomorphism with coefficients depending only on $(\varrho, \theta)$. 

Altogether, in each $\D_p$, 
\begin{equation}
\calL_t = t^{4/3} \calL_\varrho \quad \mbox{where} \qquad \calL_\varrho = \Delta_\varrho + M_\varrho. 
\label{defLt}
\end{equation}
The operator $\calL_\varrho$ is smooth on $\RR^2$, and converges exponentially quickly as $\rho \to \infty$ to 
\begin{equation}
\calL_\infty = \Delta_\infty + M_\infty;
\label{defLr}
\end{equation}
here $\Delta_\infty$ is the Laplacian for $A_\infty^{\fid}$ and $M_\infty =-2 \pi^{\skew}(i \ast [ (\Phi_\infty^{\fid})^* \wedge 
[ \Phi_\infty^\fid \wedge \cdot ]])$, both expressed in terms of $\varrho$. 

It follows from \eqref{defLt} that if we consider the operator $\calL_t$ evaluated at a fiducial 
solution $(A_t^{\fid}, \Phi_t^{\fid})$, acting on some space of fields (with specified decay) on the entire plane $\RR^2$,
then the Schwartz kernel of its inverse $G_t^{\fid}$ satisfies 
\begin{equation}
G_t^{\fid} (z, \tilde{z}) = G_\varrho (t^{2/3} z, t^{2/3} \tilde{z}).
\label{skerneltransf}
\end{equation}
(Note that we might expect an additional factor of $t^{-4/3}$ on the right side of this equation; this actually does appear because
of the homogeneity of the standard Lebesgue measure $d\sigma(\tilde{z})$ on $\CC$, cf.\ also the proof of
Proposition \ref{seriestoseries} below.) To check this, we calculate
\[
\calL_t  G_t^{\fid}( z, \tilde{z}) = t^{4/3} (\calL_\varrho G_\varrho) ( t^{2/3} z, t^{2/3} \tilde{z}) =  \\ 
t^{4/3}  \delta( t^{2/3} z - t^{2/3} \tilde{z}) = \delta(z-\tilde{z})
\]
since the delta function in two dimensions is homogeneous of degree $-2$. 

We next check that $G_t^{\fid}$ is uniformly bounded in $L^2$ for $t \geq 1$ (and indeed its norm
decreases as $t \to \infty$).  To this end, define  $(U_t f)(w) = t^{-2/3} f(t^{-2/3} w)$, so that
$U_t: L^2 ( d\sigma(z)) \to L^2 (d\sigma(w))$ is unitary for all $t$. We then write
\begin{multline*}
u(z) = G_t^{\fid} f(z) = \int G_\varrho( t^{2/3} z, t^{2/3} \tilde{z}) f(\tilde{z})\, d\sigma(\tilde{z})  \\ =
t^{-2/3} \int G_{\varrho} (t^{2/3}z , \tilde{w}) (U_t f)(\tilde{w}) \, d\sigma(\tilde{w}),
\end{multline*} 
so that
\[
(U_t u)(w) = t^{-4/3} G_{\varrho} ( U_t f) (w),
\]
or finally
\[
G_t^{\fid} = t^{-4/3} U_t^{-1} G_{\varrho} U_t,
\]
which proves the claim. 

We define $X' := X \setminus \bigsqcup_{p\in\mathfrak p} \ \mathbb D_p$ and refer to this set as the exterior region  in the following.
If $(A_\infty, \Phi_\infty)$ is the limiting configuration used in the approximate solution $S_t^{\appr}$, let $G^{\exterior}$ denote 
an inverse (or even just a parametrix up to smoothing error) for the corresponding operator $\calL_\infty$ on
the exterior region.  Writing $\D_p(a)$ for the disk of radius $a$ around $p$, choose a partition of unity $\{\chi_1, \chi_2\}$ 
subordinate to the open cover $\bigsqcup \D_p$ and $X \setminus \bigsqcup \overline{\D_p(7/8)}$. Choose 
two further cutoff functions $\tilde{\chi}_1$ and $\tilde{\chi}_2$ so that $\tilde{\chi}_j = 1$ on the support of $\chi_j$, 
and with $\mathrm{supp}\, \tilde{\chi}_1 \subset \bigsqcup \D_p$, $\mathrm{supp}\, \tilde{\chi}_2 \subset
X \setminus \bigsqcup \overline{\D_p (3/4)}$. Then define the parametrix for $\calL_t$, 
\[
\widetilde{G}_t = \tilde{\chi}_1 G_t^{\fid} \chi_1 + \tilde{\chi}_2 G^{\exterior} \chi_2.
\]
As an equation of distributions on $X \times X$,
\[
\widetilde{G}_t \calL_t = \mathrm{Id} - R_t;
\]
this remainder term
\[
R_t = \tilde{\chi}_1 G_t^{\fid} [ \calL_t, \chi_1] + \tilde{\chi}_2 G^{\exterior} [ \calL_t, \chi_2] + 
\tilde{\chi}_2 R^{\exterior} \chi_2.
\]
is a smoothing operator; indeed, the support of $\tilde{\chi}_j(z)$ does not intersect the support of $\nabla \chi_j(\tilde{z})$, $j = 1,2$, 
and the Green functions are singular only along the diagonal, so the first two terms have smooth kernels.  The remaining
term $R^{\exterior}$ is the smoothing error $G^{\exterior}\calL_t = \mathrm{Id} - R^{\exterior}$. 

Suppose now that $u_t$ and $f_t$ satisfy $\calL_t u_t = f_t$, or equivalently,
$u_t = G_t f_t$. Applying $\widetilde{G}_t$ to $f_t$ instead gives that 
\begin{equation}
u_t = \widetilde{G}_t f_t + R_t u_t.
\label{param}
\end{equation}
%The support of the Schwartz kernel $R_t(z, \tilde{z})$ is contained in the region 

We are interested in two specific mapping properties. The first one when $f_t$ is supported in the exterior region,
outside the disks, and the second when $f_t$ is supported in one of these balls and has
the form $f_t( r,\theta) = f( t^{2/3} r, \theta)$.  We consider these in turn.

\begin{proposition}
Suppose that $\calL_t u_t = f$, where $f$ is $\calC^\infty$ and supported in the exterior region
$X'$. Then for any $k \geq0$, $\|u\|_{H^{k+2}(X)} \leq C t^m\|f\|_{H^k(X)}$
where $m=m(k)>0$ and  $C$ is independent of $t$. 
\label{exttoint}
\end{proposition}
\begin{proof} 
Since $\calL_t^{-1}: L^2 \to L^2$ is bounded uniformly for $t \geq 1$, we have $\|u_t\|_{L^2} \leq C \|f\|_{L^2}$ (on
all of $X$), where $C$ is independent of $t$. Next, the coefficients of $\Delta_{A_t}=\calL_t-t^2 M_{\Phi_t}$ and of $M_{\Phi_t}$
are uniformly bounded in $\calC^\infty$ on $X'$, so employing local elliptic estimates there, and using the estimate above for
the $L^2$ norm of $u_t$ shows that $\|u_t\|_{H^{k+2}( X')} \leq Ct^2 \|f\|_{H^k(X)}$, again with $C$ independent of $t$. We turn
this estimate into one over $\D_p$ as follows. We first extend $u_t$ from $X'$ to a function $v_t$ on $X$ such that $\|v_t\|_{H^{k+2}(X)}
\leq Ct^2 \|f\|_{H^k(X)}$. In particular, the difference $w_t:=u_t-v_t$ satisfies Dirichlet boundary conditions on $\D_p$ and vanishes
on $X'$. Also, the restriction to $\D_p$ of $w_t$  satisfies $ \calL_t w_t=-\calL_t v_t$. Because the coefficients of the operator
$ \calL_t$ are polynomially bounded in $t$ it follows that  $\| \calL_t w_t\|_{H^{k}( \D_p)} \leq Ct^{m_1} \|f\|_{H^k(X)}$ for some
$m_1=m_1(k)\geq2$. Arguing now exactly as in the proof of \cite[Proposition 5.2 (ii)]{msww14}, it follows that
$\|  w_t\|_{H^{k+2}( \D_p)} \leq Ct^{m} \|f\|_{H^k(X)}$ for some further constant $m=m(k)\geq m_1$. Therefore,
$\|  u_t\|_{H^{k+2}(X)} \leq \|  w_t\|_{H^{k+2}(X)} + \|  v_t\|_{H^{k+2}(X)}  \leq  Ct^{m} \|f\|_{H^k(X)}$, proving the claim. 
\end{proof}

We now come to a key concept. The class of functions (or fields) which arise in the rest of this paper have the property that 
they decay exponentially as $t \to \infty$ away from the zeroes of $q$, but concentrate with respect to the natural dilation
near each of these zeroes.  We call the building blocks of such functions {\it exponential packets}. 
\begin{definition}\label{defexppackets}
   A family of functions  $\mu_t(z)$ on $\RR^2$ is called an {\it  exponential packet} if it is of the form
    $\mu_t(z)=(t^{2/3} |z|)^\tau \mu( t^{2/3} z)$ where
\begin{itemize}
\item $\mu_t(z) = \mu( t^{2/3} z)$ 
where $\mu(w)$ is smooth and decays like $e^{- \beta |w|^{3/2}}$ along with all of its derivatives for some $\beta > 0$,
\item	$\tau > 0$.
\end{itemize}
An {\it   exponential packet of weight} $\sigma$ is a function of the form $t^\sigma \mu_t(z)$, where $\sigma \in \RR$ and $\mu_t(z)$ is 
an exponential packet. Finally, we say simply that a function $\mu_t$ on $X$ is a {\it convergent sum of exponential 
packets} if in the standard holomorphic coordinate in each $\D_p$ it is a $\calC^\infty$ convergent sum of exponential packets
and decays like $e^{-\beta t}$ for some $\beta > 0$ along with all its derivatives outside of the $\D_p$.  If 
the exponential packets involve factors of $(t^{2/3}|z|)^\tau$ as above, then the sense in which these sums converge
must be modified. In the applications below we shall only encounter the same extra factor $(t^{2/3} |z|)^{1/2}$ in
all terms of the sum, so it may be simply pulled out of the sum.

%A family of functions $\mu_t(z)$ on $\RR^2$ is called an exponential packet if it is of the form $\mu_t(z) = \mu( t^{2/3} z)$ 
%where $\mu(w)$ is smooth and decays like $e^{- \beta |w|^{3/2}}$ along with all of its derivatives for some $\beta > 0$.  
%Slightly more generally, we shall also encounter families of the form $(t^{2/3} |z|)^\tau \mu( t^{2/3} z)$ where $\mu$
%is smooth and exponentially decreasing and $\tau > 0$. We refer to these too as exponential packets. 
%A weighted exponential packet is a function of the form $t^\sigma \mu_t(z)$, where $\sigma \in \RR$ and $\mu_t(z)$ is 
%an exponential packet. Finally, we say simply that a function $\mu_t$ on $X$ is a (convergent) sum of exponential 
%packets if in the standard holomorphic coordinate in each $\D_p$ it is a $\calC^\infty$ convergent sum of expontial packets
%and decays like $e^{-\beta t}$ for some $\beta > 0$ along with all its derivatives outside of the $\D_p$.  If 
%the exponential packets involve factors of $(t^{2/3}|z|)^\tau$ as above, then the sense in which these sums converge
%must be modified. In the applications below we shall only encounter the same extra factor $(t^{2/3} |z|)^{1/2}$ in
%all terms of the sum, so it may be simply pulled out of the sum.
\end{definition}
\begin{proposition}
Suppose that $f_t(z)$ is an exponential packet supported in some $\D_p$.  Then $u_t = G_t f_t $ is an  exponential packet 
$t^{-4/3} \mu_t( t^{2/3} z)$ 
\label{seriestoseries}
of weight $-\frac{4}{3}$.
\end{proposition}
\begin{proof}
%Denoting the area form by $\sigma$, 
We have 
\[
\int G_t^{\fid}(z, \tilde{z}) f( t^{2/3} \tilde{z}) \, d\sigma(\tilde{z}) = t^{-4/3} \int G_t^{\fid} (z, t^{-2/3} \tilde{w}) 
f(\tilde{w}) \, d\sigma(\tilde{w}).
\]
Thus if we set $w = t^{2/3} z$, then the right hand side equals
\[
t^{-4/3} \int G_t^{\fid} (t^{-2/3} w, t^{-2/3} \tilde{w}) f(\tilde{w})\, d\sigma(\tilde{w})|_{w = t^{2/3} z}  =  t^{-4/3}\mu_t(z). 
\]
This computation shows that $G_t^{\fid} f_t$ is exponentially small outside of $\D_p(1/2)$, say. 

Now fix a cutoff function $\chi$ which equals $1$ in $\D_p(3/4)$ and which vanishes outside $\D_p(7/8)$, and set 
$\tilde{u}_t = \chi G_t^{\fid} f_t$. (In other words, we localize the function $G_t^\fid f$ from $\RR^2$ to the disk.)   Then
\[
\calL_t (\tilde{u}_t - u_t) = [ \calL_t, \chi] G_t^{\fid} f_t + \chi f_t - f_t := h_t.
\]
The calculation above shows that $h_t$ decays exponentially.  Hence writing $u_t = \tilde{u}_t - v_t$, then $v_t = G_t h_t$ 
decays exponentially, first in any Sobolev norm, then in $\calC^\infty$.  This proves the result. 
\end{proof}

The preceding results now give the following useful result:
\begin{corollary}
  If $f_t$ is a convergent sum of exponential packets, then $u_t = G_t f_t$ is also a convergent sum
  of exponential packets. More precisely, 
  \[
    f_t = \sum_j t^{\sigma - 2j/3}f_{j,t}  + \calO(e^{-\beta t}) \Longrightarrow
    u_t = \sum_j t^{\sigma -4/3 - 2j/3}u_{j,t}  + \calO(e^{-\beta t}).
  \]
  \label{convergence}
\end{corollary}

%We record a final useful calculation.
%\begin{lemma}
%If $t^s F_t(z)$ is an  exponential packet of weight $s$, then 
%\[
%\int |t^s F_t(tz)|^2 \, d\sigma(z) = t^{2s - 4/3} \int |F_t(w)|^2 \, d\sigma(w)
%\]
%\end{lemma}

\subsection{Smooth dependence on parameters}
The considerations above will be applied in the next sections to prove the existence of expansions as $t \to \infty$
for the various components of the $L^2$ metric.  An important addendum is that these are true polyhomogeneous
expansions, i.e., the derivatives with respect to various parameters of these metric coefficients
have the corresponding differentiated expansions.  For certain derivatives, e.g.\ those with respect to $t$, this
is not hard to deduce. However, it is much less obvious for derivatives in other directions, particularly those
with respect to $q$. We now discuss the reasoning which will lead to this conclusion in all cases. 

The first key point is the fact that the spectral curve $S_q$ varies smoothly as $q$ varies in $\calB'$. This follows
immediately from the nonsingularity of the defining relation $\lambda_{\SW}^2 - q = 0$ when $q$ lies away from the 
discriminant locus.  We have also already described the normal vector field $N_{\dot q}$ arising from the variation 
$S_{q+ s \dot q}$. It is evident from the discussion in \S 2.3 that $N_{\dot q}$ is tangent to the zero section $\underline{\bf 0}$
of $K_X$ at the intersection points $S_q \cap \underline{\bf 0}$, i.e., at the zeroes of $q$. 

The second key point is that the (sums of) exponential packets encountered below are mostly of a very special
type in that they lift to restrictions to $S_q$ of globally defined functions on $K_X$ which decay exponentially along
the fibers.  To make this precise, we define the class of {\it global} exponential packets and their sums.
By definition, a sum of global exponential packets is a function $\mu$ on the total space of $K_X$ which is
smooth away from the zero section, has an integrable polyhomogeneous singularity at $\underline{\bf 0}$,
and decays exponentially as $|w| \to \infty$ in each fiber of $K_X$.  The last two conditions here mean that in
standard coordinates $(z,w)$ on $K_X$, $\mu(z,w) \sim \sum \mu_j(z, \arg w) |w|^{\gamma_j}$ as $w \to 0$, where
each $\mu_j$ is smooth and the exponents $\gamma_j \to \infty$, and $|\mu(z,w)| \leq Ce^{-\beta |w|}$ as $w
\to \infty$.  (The examples here are all of the form $\gamma_j = j$ or $\gamma_j = j + 1/2$, $j \in \mathbb N$.)  

\begin{proposition}\label{global_exp_packet}
Let $\mu$ be a convergent sum of global exponential packets on $K_X$ and $\mu_q$ the restriction of $\mu$ to the spectral 
curve $S_q$. Then the family of integrals 
\[
q \longmapsto \int_{S_q} \mu_q\, dA
\]
has a convergent expansion as $\|q\|_{L^2} \to \infty$ in $\calB'$, which holds along with all its derivatives.
\end{proposition}
\begin{proof}
Let $q$ vary along a transversal to the $\RR^+$ action and consider the function
\[
(t, q) \longmapsto \int_{S_{tq}} \mu_{tq}\, dA = \int_{t S_q} \mu_{tq}\, dA.
\]
The restrictions of these integrals to any fixed region $|w| \geq c > 0$ in $K_X$ decay exponentially in $t$,
uniformly as $q$ varies in a small set. Thus we may restrict to disks $\D_i$ in $S_q$ centered
at the zeroes of $q$ and write the corresponding integrals in local coordinates.   For $q$ fixed,
the integral of an exponential packet on a fixed disk is a monomial $c t^\alpha$ for some $\alpha$,
so the integral of a convergent sum of exponential packets becomes a convergent sum of
such monomials. This is clearly polyhomogeneous in $t$.  The smoothness in $t$ is also straightforward
from these local coordinate expressions.

The smoothness in $q$ is also now clear since the spectral curve varies smoothly with $q$. There is
one small point to mention, however. If $\mu$ has a polyhomogeneous singularity along the zero section,
we must use that the variation of $S_q$ is tangent to the zero section. Indeed, we can write the contribution
on the disk around $q$ as an integral on a varying family of disks transverse to the zero section in $K_X$.
The derivative of this integral with respect to $q$ is then the integral of the derivative of $\mu$
with respect to the variation vector field.  However, $\mu$ is polyhomogeneous along the zero section
so differentiating it with respect to vector fields tangent to the zero section does not change its
regularity nor the form of its asymptotic expansion at the zero section. This implies that
the derivative in $q$ of the integral along this family of disks is smooth in $q$.
\end{proof}

\section{Horizontal asymptotics of the $L^2$-metric}\label{horizontal}
In this and the next few sections, we put into gauge the infinitesimal deformations of the families of approximate 
solutions, and then evaluate the $L^2$ metric on these.   We begin now by considering 
the horizontal tangent vectors on $(\mathcal M^{\appr})'$. 

Henceforth, fix an approximate solution
\[
S_t^{\appr} = (A_t^{\appr} + \eta, t\Phi_t^\appr) \in (\mc M^\appr)'. 
\]
Now consider the variations of \eqref{eq:atappr} and \eqref{eq:phitappr} with respect to $q$: 
\begin{equation}\label{eq:dotAapp}
\begin{split}
\dot{A}_t^{\appr} & :=\left.\frac{d}{d\varepsilon}\right|_{\varepsilon=0}A_t^{\appr}(q+\varepsilon \dot q) \\ 
& =\Bigl( 4   f_t'(|q|_k) |q|_k \Re\frac{\dot q}{q} \Im \bar\partial \log |q|_k -2 f_t(|q|_k) d \Im\frac{\dot q}{q}\Bigr) 
\begin{pmatrix}  i & 0 \\ 0 & -i
 \end{pmatrix}. 
\end{split}
\end{equation}
and 
\begin{equation}\label{eq:dotPhiapp}
\dot{\Phi}_t^{\appr}:=\left.\frac{d}{d\varepsilon}\right|_{\varepsilon=0}\Phi_t^{\appr}(q+\varepsilon \dot q)=
\begin{pmatrix}
0 & e^{-h_t(|q|_k)}  |q|_k^{-\frac{1}{2}} (\dot q- q Q)  \\
e^{h_t(|q|_k)}|q|_k^{1/2}  Q  & 0
\end{pmatrix} , 
\end{equation}
where $Q = \frac12 + |q|_kh_t'(|q|_k) \Re \frac{\dot q}{q}$. Then $(\dot{A}_t^{\appr}+\dot\eta,t\dot{\Phi}_t^{\appr})$, 
$\dot \eta=[\eta\wedge\gamma_{\infty}]$, is tangent to $(\mc M^{\appr})'$ at $S_t^{\appr}$, cf.~Lemma \ref{lem:tangdirectionslimconf}. 

%The pair $(\dot A_t^{\appr}+\dot\eta,\dot{\Phi}_t^{\appr})$ is now a tangent to $\mathcal M^{\appr}$ at $(A_t^{\appr}, \Phi_t^{\appr})$. 

The gauge-correction is a two-step process. First we employ an infinitesimal gauge-transformation 
adapted to the local structure of $S_t^{\appr}$ near the zeroes of $q$. The remaining correction term is found 
using the global methods from \S 5. 

\subsection{Initial gauge correction step} 
The infinitesimal gauge transformation
\[
\gamma_t := -2 f_t(|q|_k) \Im \frac{\dot q}{q} \, \begin{pmatrix} i & 0 \\ 0 & -i \end{pmatrix}
\]
is the obvious desingularization of the field $\gamma_\infty$ used in \S 3 to remove the main singularity
of the limiting configuration.  We thus define 
\begin{equation*}
(\alpha_t,t\varphi_t):=  (\dot A_t^{\appr}+\dot \eta,  t \dot \Phi_t^{\appr})-D^1_{S_t^{\appr}}\gamma_t \in T_{S_t^{\appr}} \mc M^{\appr},
\end{equation*}
or more explicitly, 
\begin{equation}\label{eq:hortangvecgauge}
\begin{split}
\alpha_t &:=\dot A_t^{\appr}+\dot \eta-d_{A_t^{\appr}+\eta}\gamma_t, \\
t\varphi_t &:=t\dot \Phi_t^{\appr}-t[\Phi_t^{\appr}\wedge\gamma_t].
\end{split}
\end{equation}

This is a tangent vector to a small perturbation of a point in $(\mc M^{\app})'$ at radius $t$, so it is
natural to rescale this tangent vector by a factor of $t$ and show that it converges as $t \to \infty$.
In other words, we consider convergence of the pair $(t^{-1} \alpha_t, \varphi_t)$. 
Since $\gamma_t \to \gamma_\infty$ in $\calC^\infty$ away from the zeroes of $q$, we see that 
\[
(t^{-1} \alpha_t,\varphi_t) \to (0, \varphi_\infty) = (\dot A_\infty, \dot \Phi_\infty) - D^1_{S_\infty} \gamma_\infty 
\quad  \mbox{as}\  \ t \to \infty.
\]
(In fact, $\alpha_t$ tends to $0$ away from each $\D_p$ even without the extra factor of $t^{-1}$.) 
Direct calculation shows that this pair is closer by a factor $t^{-m}$, $m>0$, to being in gauge
than $(\dot A_t^{\appr}, t \dot \Phi_t^{\appr})$.}

We now examine $\alpha_t$ and $\varphi_t$ more closely.   First,
\begin{equation*}
d_{A_t^{\appr}+\eta}\gamma_t=[\eta\wedge\gamma_t]-2\Big(f_t'(|q|_k)\Im \frac{\dot q}{q} d |q|_k + f_t(|q|_k) d \Im  \frac{\dot q}{q} \Big) 
\begin{pmatrix}i&0\\0&-i\end{pmatrix},
\end{equation*}
whence, recalling that $\dot\eta=[\eta\wedge\gamma_{\infty}]$, 
\begin{equation}\label{eq:correctedhortangalpha}
\begin{split}
\alpha_t &= \dot A_t^{\appr} + \dot\eta -d_{A_t^{\appr}+\eta}\gamma_t\\
 &=[\eta\wedge(\gamma_{\infty}-\gamma_t)] + 4 f_t'(|q|_k)\Im \frac{\dot q}{q}  \, d |q|_k \begin{pmatrix}i&0\\0&-i\end{pmatrix}.
\end{split}
\end{equation}
As for the other term, 
\begin{equation*}
[\Phi_t^{\appr}\wedge\gamma_t]=4if_t(|q|_k)\Im\frac{\dot q}{q} \begin{pmatrix}0&|q|_k^{-\frac{1}{2}}e^{-h_t(|q|_k)q}\\
-|q|_k^{\frac{1}{2}}e^{h_t(|q|_k)}&0\end{pmatrix},
\end{equation*}
so that
\begin{equation}\label{eq:correctedhortangentphi}
\begin{split}
	\varphi_t &= \dot\Phi_t^{\appr}-[\Phi_t^{\appr}\wedge\gamma_t]\\
&=\begin{pmatrix} 0&\Big( \frac{1}{2}-|q|_kh_t'(|q|_k)\Big)e^{-h_t(|q|_k)}|q|_k^{-\frac{1}{2}}\dot q\\
\Big( \frac{1}{2}+|q|_kh_t'(|q|_k)\Big)e^{h_t(|q|_k)}|q|_k^{\frac{1}{2}}\frac{\dot q}{q}&0   \end{pmatrix} \, dz.
\end{split}
\end{equation}

We next analyze the asymptotics of the family $(t^{-1} \alpha_t, \varphi_t)$ in each disk $\D_p$. 
\begin{proposition}\label{prop:normalphatphi}
Fix $\varphi_{\infty}\neq0$ as in \eqref{eq:varphiinfty}. Then in each disk $\D_p$, % where $q(p) = 0$, 
\[
t^{-1} \alpha_t = \sum_{j=0}^\infty  A_{j,t} t^{(1-2j)/3}
\]
and
\[
\varphi_t - \varphi_\infty =  \sum_{j=0}^\infty B_{j,t} t^{ (1-2j)/3} 
\]
as $t \to \infty$, where the coefficients $A_{j,t}$ and $B_{j,t}$ are exponential packets  and the sum is convergent. 
Outside the union of the disks $\D_p$,  
\[
|t^{-1} \alpha_t| + |\varphi_t - \varphi_\infty| \leq C e^{-\beta t}.
\]
\end{proposition}
\begin{proof}
The exponential decay outside the $\D_p$ is clear, so we focus on the behavior inside one of the disks. With a holomorphic 
coordinate $z$ for which $q= z dz^2$, we have $\dot q = \dot f dz^2$ for some holomorphic $\dot f$. We assume further that $H$ 
is the standard flat metric on the local holomorphic frame $dz^{\pm 1/2}$ and that $\eta$ vanishes on $\D_p$.  Then in this region,
\begin{equation}\label{alphainDp}
\begin{split}
& \alpha_t = 4 f_t'(r) \Im \frac{\dot f}{z} dr \begin{pmatrix} i & 0 \\ 0 & -i \end{pmatrix}, \ \ \mbox{and} \\ 
\varphi_t - &\varphi_\infty   =  \\
& \begin{pmatrix}
 	0 & \bigl( ( \frac{1}{2} - r h_t'(r) ) e^{-h_t(r)} - \frac{1}{2} \bigr) r^{-\frac{1}{2}} \dot f\\
 	\bigl( ( \frac{1}{2} + rh_t'(r) ) e^{h_t(r)} -\frac{1}{2} \bigr) r^\frac{1}{2} \frac{\dot f}{z} & 0 
 \end{pmatrix} dz.
\end{split}
\end{equation}
We now recall that $f_t$, $h_t$ and $(r\del_r) h_t$ are all functions of $\rho = t r^{3/2}$ and satisfy $f_t(\rho) \to 1/8$
and $h_t(\rho) \leq C e^{-\beta \rho}$.  A brief calculation shows that $f_t'(r) $ is $t^{2/3}$ times a smooth exponentially 
decreasing function of $\rho$. The assertions now follow once we expand $\dot f$ in a Taylor series and 
write each $r^j$ as $(t^{2/3} r)^j t^{-2j/3}$ in the expression for $\alpha_t$ and 
$r^{j-1/2} = (t^{2/3} r)^{j-1/2} t^{(1 -2 j)/3}$ in the expression for $\varphi_t - \varphi_\infty$. 
\end{proof}

We briefly describe the regularity of the coefficients in \eqref{alphainDp} when pulled back to the spectral curve. 

First, up to constant multiples, the coefficients in $\alpha_t$ have the form 
\[
f_t'(|q|_k) \Im \left(\frac{\dot q}{q}\right) d|q|_k = f_t'(|\lambda|^2) \Im \left(\frac{\dot q}{\lambda^2}\right) d|\lambda|^2
\]
where we consider the right side as a function of $\lambda \in K_X$.  However, $f_t(r)$ has a double zero, hence $f_t'(r)$ 
vanishes at $r=0$, so $f_t'(|\lambda|^2)$ vanishes to order $2$, and altogether this expression has a simple zero
at the zero section. 

On the other hand, the upper right coefficient in $\varphi_t - \varphi_\infty$ has the form
\[
\mu_t(|q|_k)|q|_k^{-1/2}\dot q = \frac{\mu_t(|\lambda|^2)}{|\lambda|} \dot q,
\]
where $\mu_t$ is an exponential packet. This has a simple pole at the zero section of $K_X$, and as we now check, its restriction
to the spectral curve is bounded. Indeed, choose the usual coordinate $w^2 = z$, so $\dot q = \dot f dz^2 = 4 \dot f w^2 dw^2$ and $\lambda= w dz = 2 w^2 dw$. These give that $\dot q/|\lambda|= 2 \dot f \frac{w^2}{|w|^2|dw|}dw^2$.  The discussion 
for the coefficient in the lower left is analogous. 

In either case, the terms are global exponential packets of precisely the sort considered in Proposition \ref{global_exp_packet}.

% \begin{corollary}\label{cor:horrateofconv}
% \[
% \|t^{-1} \alpha_t\|^2_{L^2(X)} + \|\varphi_t - \varphi_\infty\|_{L^2(X)}^2 \sim \sum_{j =0}^\infty C_j t^{ (-2 - 2j)/3}.
% \]
% as $t \to \infty$, in particular $\|t^{-1} \alpha_t\|^2_{L^2}+\|\varphi_t - \varphi_\infty\|^2_{L^2}=\calO(t^{-2/3})$. 
% \end{corollary}

% \begin{proof} 

% An important point in this last estimate is that because the error terms are all exponentially small,
% we can in principle calculate each coefficient $C_\ell$ explicitly since it equals a sum over the
% zeroes of $q$ of integrals $\int A_i B_j$ with $i + j = \ell$. In particular, we can ensure that infinitely 
% many of these terms are nonzero. 

\subsection{Second gauge correction step}
Following \eqref{eq:gaugefixeq}, we now solve 
\begin{equation}
\calL_t \xi_t = R_t:=d_{A_t^{\appr}+\eta}^{\ast}\alpha_t - 2t^2\pi^{\skew}(i \ast [(\Phi_t^{\appr})^* \wedge \varphi_t]).
\label{gaugefix2}
\end{equation}
\begin{lemma} \label{prop:errorest}
The error term $R_t$ is a convergent sum of  exponential packets of weights $2- 2j/3$: in each $\D_p$,
\[
R_t = \sum_{j=0}^\infty t^{2- 2j/3} k_{j,t}(z) \begin{pmatrix} i & 0 \\ 0 & -i \end{pmatrix}, \qquad k_{j,t}(z) = k_j(t^{2/3} z).
\]
\end{lemma}
\begin{proof} 
As before, choose a holomorphic coordinate $z$ in $\D_p$ so that $q = z dz^2$, and assume that hermitian metric is
trivial on the frame $dz^{\pm 1/2}$. Following the discussion in \S 4, assume also that $\eta$, and hence 
$\dot \eta = [\eta \wedge \gamma_\infty]$, both vanish on $\D_p$. 

Using \eqref{alphainDp}, we calculate that
\begin{align*}
d_{A_t^\appr}^* \alpha_t &= 4 \, d^* \bigl( f_t'(r) \Im(\dot f/z)  dr \bigr) \, \begin{pmatrix} i & 0 \\ 0 & -i \end{pmatrix}, \\
& = 4\bigl( - \partial_r (f_t'(r)r^{-1}) - f_t'(r)r^{-2}  - (f_t'(r)r^{-2}) r\partial_r\bigr) \Im (e^{-i\theta}\dot f)
\begin{pmatrix} i & 0 \\ 0 & -i \end{pmatrix}.
\end{align*}
This can then be simplified using 
\begin{multline*}
f_t'(r)r^{-2} = 2t^2 \sinh(2h_t(r)), \ \ \mbox{and} \\
\del_r( f_t'(r)r^{-1}) = \del_r( 2 t^2 r \sinh (2 h_t(r)) ) = 2t^2 ( 1 + r\del_r) \sinh (2h_t(r)),
\end{multline*}

In addition, 
\begin{multline*}
-2t^2 \pi^{\skew}(i \ast [(\Phi_t^{\appr})^* \wedge \varphi_t ]) = \\ 
4 t^2 \Re (ie^{-i\theta}\dot f) \left(\sinh(2h_t) +2(r\del_r h_t)\cosh(2h_t)\right) \begin{pmatrix} i & 0 \\ 0 & -i \end{pmatrix}.
\end{multline*}

The rest of the argument is exactly as in the proof of \eqref{prop:normalphatphi}.
\end{proof}

We now invoke the detailed mapping properties for $\calL_t^{-1} = G_t$ from Propositions ~\ref{exttoint} and ~\ref{seriestoseries}
 and Corollary~\ref{convergence}  to conclude the following. 
\begin{proposition}
The gauge correction field $\xi_t$ is a convergent sum of exponential packets plus an exponentially small remainder term: 
\[
\xi_t  = \sum_{j=0}^\infty  \xi_{j,t}(z) t^{(2-2j)/3} + \calO(e^{-\beta t}), \qquad \xi_{j,t}(z) = \chi_j(t^{2/3}z),
\]
and hence the actual gauge correction term $D^1_t \xi_t$ is also of this type: 
\begin{equation}
D_t^1 \xi_t = \sum_{j=0}^\infty  \eta_{j,t} (z) t^{(4-2j)/3} + \calO(e^{-\beta t}), \qquad \eta_{j,t}(z) = \eta_j(t^{2/3} z).  
\label{actualgaugefixing}
\end{equation}
\label{gaugecorrectionhorizontal}
\end{proposition}

Note that we must also include the scaling by $t^{-1}$, i.e., the gauge correction of $(t^{-1} \alpha_t, \varphi_t)$ 
is $t^{-1} D_t^1 \xi_t$, which is a sum of exponential packets starting with $t^{1/3} \eta_{0,t}$. 

The relationship between the gauged infinitesimal deformations to the approximate moduli space and to
the space of limiting configurations is then 
%expression which represents the deviation of the $L^2$ metric from the semiflat metric on $(\mc M^{\app})'$ (rescaled
%by $t^{-2}$) is
\begin{equation}
\label{difference0}
(t^{-1}\alpha_t,\varphi_t)-t^{-1}D_t^1\xi_t  =  (0,\varphi_\infty)+\sum^\infty_{j=0} C_j t^{(1-2j)/3} + \calO(e^{-\beta t}),
\end{equation}
and hence
\begin{equation}
\label{difference}
\begin{aligned}
\|(t^{-1}\alpha_t, &\varphi_t)  -t^{-1}D_t^1\xi_t\|_{L^2}^2 \\
& = \|\varphi_\infty\|^2_{L^2} +2 \langle \varphi_\infty, \sum^\infty_{j=0} C_j t^{(1-2j)/3}\rangle_{L^2} + 
\|\sum^\infty_{j=0} C_j t^{(1-2j)/3}\|_{L^2}^2  + \calO(e^{-\beta t})\\
& = \|\varphi_\infty\|^2_{L^2} + \sum_{j=0}^\infty  S_j t^{-(2+j)/3} + \calO(e^{-\beta t}). 
\end{aligned}
\end{equation}
 The shift by the factor $t^{-4/3}$ in the final series is due to the Jacobian
factor in the integration.  This same shift appears several times below. 

\medskip

This is the equation which expresses the difference between the metric coefficients for the Hitchin and semiflat metrics
in this particular direction. By polarization we can obtain a similar expansion for the mixed horizontal metric coefficients.
Thus, if $(v^{\hor})^{(j)} =  (\dot A_\infty^{(j)} + \dot \eta^{(j)},  \dot \Phi_\infty^{(j)} - D_t^1(\gamma_t^{(j)} + \xi_t^{(j)}))$, 
$j=1,2$, are two different gauged horizontal deformations, then 
\begin{multline*}
t^{-2} \langle (v^{\hor})^{(1)}, (v^{\hor})^{(2)}\rangle_{L^2}   \\ = t^{-2} \langle (v^{\hor})^{(1)}, (v^{\hor})^{(2)}\rangle_{\sfl} + 
\sum_{j=0}^\infty S_j'((v^{\hor})^{(1)},(v^{\hor})^{(2)})  t^{-(2+j)/3},
\end{multline*} 
where the $S_j'$ are symmetric $2$-tensors on horizontal tangent vectors which are independent of $t$.  

Proposition \ref{global_exp_packet} ensures that all expansions here may be differentiated, so that these are
`classical' expansions (cf.\ the discussion preceding Lemma 4.1) for the horizontal part of the metric. 

Observe from Propositions \ref{prop:normalphatphi} and \ref{gaugecorrectionhorizontal} that the two terms
$( t^{-1} \alpha_t, \varphi_t - \varphi_\infty)$ and $t^{-1} D_t^1 \xi_t$ are both sums of exponential packets with
the {\it same} leading order exponent $t^{1/3}$. This leaves open the possibility of some unexpected cancellations,
so that $S_0$ and perhaps some or all of the remaining $S_j$ might vanish.  

As already mentioned in the introduction, it has emerged in very recent work by David Dumas and Andy Neitzke that 
this cancellation actually does occur, at least along the Hitchin section and in horizontal directions.  Their paper \cite{dn18} 
presents a beautiful formula which proves that the integral expressing the difference between the semiflat and Hitchin
metrics for the model case of the Hitchin section over $\CC$ actually vanishes. This relies on a very interesting
integral identity, the full meaning of which is not yet clear.  
These authors go on to prove that the rate of convergence for the horizontal metric coefficients over the Hitchin section
on a general surface $X$ is exponential.
%While we now have hope to be able to prove this for all metric coefficients, a number
%of obstacles remain. We expect to return to this matter in the very near future. 

\section{Asymptotics in the radial direction}\label{sect:asympradial}

Amongst the horizontal directions, already analyzed in \S 6, the radial direction is distinguished. This is, of course, the direction where $\dot{q} = q$, so in particular the term $\dot{q}/q$ appearing in many formul\ae\ in that section equals $1$.

Let $(A_\infty + \eta, \Phi_\infty)$ be a limiting configuration associated with $q$ (normalized so that $\int_X |q| = 1$), 
and $(A_t^\appr + \eta,  \Phi_t^\appr)$ the corresponding family of approximate solutions. Then from \eqref{eq:varphiinfty} 
and the fact that $\Im (\dot q/q) = 0$, we obtain $\alpha_t = 0$,
\[
\varphi_\infty = \begin{pmatrix}
0 & \frac 12 |q|_k^{-1/2} q  \\ \tfrac{1}{2} |q|_k^{1/2}  & 0	
\end{pmatrix}\, dz = \frac{1}{2} \Phi_\infty,  
\]
and by  \eqref{eq:correctedhortangentphi}, 
\[
\varphi_t= \dot \Phi_t^\appr = \begin{pmatrix}
 0 & ( \tfrac{1}{2} - |q|_k h_t'(|q|_k)) e^{-h_t(|q|_k)}  q/ |q|_k^{1/2}  \\
( \tfrac{1}{2} + |q|_k h_t'(|q|_k))  e^{h_t(|q|_k)} |q|_k^{1/2} & 0	
\end{pmatrix}.
\]

Subtracting these, or more simply using $\dot q/q = 1$ in \eqref{alphainDp}, we have
\begin{multline*}
\varphi_t - \varphi_\infty   =  \\
\begin{pmatrix}
 	0 & \bigl( ( \frac{1}{2} - |q|_k h_t'(|q|_k) ) e^{-h_t(|q|_k)} - \frac{1}{2} \bigr) q/|q|_k^{1/2}\\
 	\bigl( ( \frac{1}{2} + |q|_kh_t'(|q|_k) ) e^{h_t(|q|_k)} -\frac{1}{2} \bigr) |q|_k^\frac{1}{2}  & 0 
 \end{pmatrix} dz.
\end{multline*}

Previously, the infinite Laurent expansion of $\dot q/q$ led to an infinite sum of weighted exponential packets, while 
here each of the two nonzero entries in $\varphi_t  - \varphi_\infty$ is a single weighted exponential packet.

We summarize these observations in the following proposition:

\begin{proposition}\label{lem:rateconvradial}
This difference has the form 
\[
\varphi_t - \varphi_\infty= \varphi_t - \tfrac 12 \Phi_\infty  = 
t^{-1/3}  \begin{pmatrix}  0 &  A_t^1(z) \\ B_t^1(z) & 0 \end{pmatrix} 
+ \calO( e^{-\beta t}),
\]
where the two off-diagonal terms are  exponential packets of weight $-1/3$. 
\end{proposition}

Next we put $(0, \varphi_t)$ into Coulomb gauge by solving
\[
\calL_t\xi_t=R_t :=- 2t \pi^{\skew}(i\ast[(\Phi_t^{\appr})^{\ast}\wedge \varphi_t]).
\]
Notice that the approximate solution $(A_t^\appr + \eta, t\Phi_t^\appr)$ is altered by $\varphi_t$, rather than $t\varphi_t$, 
which explains why there is only the single factor  $t$ rather than the factor $t^2$ in \eqref{gaugefix2}. %Following the 
%proof of Proposition \ref{prop:errorest}, there is an extra decaying term $\Re(e^{-i \theta}\dot f) = r = c t^{-2/3}\varrho^{2/3}$.

Inspecting the terms above, and also rewriting $\varphi_t = \varphi_\infty + (\varphi_t - \varphi_\infty)$,
$\Phi_t^\fid = \Phi_\infty + (\Phi_t^\fid - \Phi_\infty)$ so as to take advantage of normality at $t=\infty$, we obtain
\begin{proposition}\label{lem:radialerrorbound}
The error term $R_t$ is a diagonal exponential packet of weight $1/3$,
\[
R_t  =  t^{1/3} H_{t} %+ t^{-4/3} H_{2,t} 
+ \calO(e^{-\beta t}).
\]
Consequently, $\xi_t =  t^{-1} J_{t} + \calO(e^{-\beta t})$ and $D_t^1 \xi_t = t^{-1/3} K_{t} + \calO(e^{-\beta t})$, where $J_t$ and $K_t$ are exponential packets. 
\end{proposition}

\begin{proposition}
The $L^2$ metric on a radial tangent vector has the following expansion:
\[
\| (0,\varphi_t)-D_t^1\xi_t\|_{L^2}^2 = \|\varphi_\infty\|_{L^2}^2 + a t^{-5/3}   + \mathcal{O}(e^{-\beta t}).
\]
%The difference between the $L^2$ metric and the semiflat metric on this radial tangent vector is 
%\begin{multline*}
%\[
%\| \varphi_t - \varphi_\infty - D_t^1 \xi_t\|^2_{L^2}  \\
%=  P_t  t^{-2} + \calO( e^{-\beta t}).
%\]
\end{proposition}
%{\bf HW Also here we should write
%\[
%(0,\varphi_t)-D_t^1\xi_t = (0,\varphi_\infty) + t^{-1/3}  (0,\begin{pmatrix}  0 &  A_t^1(z) \\ B_t^1(z) & 0 \end{pmatrix} )-  t^{-4/3} K_{t} + \calO(e^{-\beta t})
%\]
%so that
%\[
%\| (0,\varphi_t)-D_t^1\xi_t\|_{L^2}^2 = \|\varphi_\infty\|^2 + a t^{-5/3} + b t^{-6/3} + c t^{-6/3} + d t^{-8/3} + e t^{-9/3} + f t^{-12/3} + \mathcal{O}(e^{-\beta t})
%\]
%}

Note finally that by \eqref{homogeneity}, $\|\varphi_\infty\|_{L^2}^2 = \|tq\|_{\sK}^2$ at the point $t^2 q$, 
and this equals $1/4$ provided $\int_X |q| = 1$.

\section{Asymptotics in fiber directions}
We now consider variations in the fiber directions. Just as in the previous section, we first compute the infinitesimal
deformations of approximate solutions, and then use a similar two-step correction to put these into gauge. 

Fix a limiting configuration which, to simplify notation, we write simply as $(A_{\infty},\Phi_{\infty})$ rather than 
$(A_{\infty}+\eta,\Phi_{\infty})$, even though it is not necessarily in the Hitchin section.  By Proposition \ref{corresp} and Corollary 
\ref{L2_repr}, a fiberwise infinitesimal deformation of $(A_\infty, \Phi_\infty)$ is an element of $H^1(X^{\times};L_{\infty})$, which in turn is identified with a unique $L^2$ harmonic representative in 
\begin{equation*}
\mathcal H^1(X^{\times};L_{\infty})=\{\alpha_\infty \in \Omega^1(X^\times,L_\infty) : p_q^* \alpha_\infty \in \mathcal{H}^1(S_q; i \R)_{\mathrm{odd}} \}, 
\end{equation*}
where $p_q: S_q \to X^\times$ is the spectral cover. We use the notation that the complex line bundle $L_{\infty}^{\C} =
\{\gamma\in\mathfrak{sl}(E) \mid [\Phi_{\infty} \wedge\gamma]=0\}$ on $X^ {\times}$ splits into the sum of real line bundles, 
$L_{\infty}:=L_{\infty}^{\C}\cap \mathfrak{su}(E)$ and $iL_{\infty}$, of skew-hermitian and hermitian elements, respectively. 

We first replace this infinitesimal deformation with one supported in the union of annuli $\mathbb A_p := \D_p \setminus \D_p(1/2)$. 
\begin{lemma}\label{lem:primitive}
For each $\alpha_\infty \in \mathcal H^1(X^{\times};L_{\infty})$, there exists 
$\xi_\infty \in \Omega^0(X^\times, L_\infty)$ with $\supp \xi_\infty \subset \bigsqcup_{p \in \mathfrak{p}} \D_p$ and 
$\xi_\infty(z)  = \sum_{j=0}^\infty \xi_{\infty,j} r^{j + 1/2}$ near each $p$, so that  
\begin{equation}\label{eq:defbetainfty}
\beta_{\infty}:=\alpha_{\infty} - d_{A_{\infty}}\xi_{\infty}
\end{equation}
is supported outside each $\D_p(1/2)$. Furthermore, $d_{A_{\infty}}\beta_{\infty}=0$ and $R:=d_{A_{\infty}}^{\ast}\beta_{\infty}\in  
\Omega^0(X^{\times}; L_{\infty})$ is supported in $\bigsqcup_{p\in\mathfrak p}\mathbb{A}_p$.
\end{lemma}
\begin{proof}
Choose coordinates on each $\D_p$ such that $z=p_q(w)=w^2$. Then $p_q^* \alpha_\infty = f dw - \bar f d \bar w$ where $f$ is holomorphic and 
even  with respect to the involution $\sigma(w)=-w$ (observe that $dw, d\bar w$ are odd with respect to $\sigma$). 
We then choose a local primitive $F(w)$ for $f(w)$, which by replacing $F(w)$ by $(F(w)-F(-w))/2$ we may as well assume to be odd, and
this then gives a local primitive $\xi_\infty|_{\D_p}$ for $\alpha_\infty$ by 
\[
p_q^* \xi_\infty|_{\D_p} = F - \bar{F}.
\] 
Since $F$ is odd, $|F(w)|=\mathcal{O}(|w|)$, so $\xi_\infty|_{\D_p} = \mathcal{O}(r^{1/2})$.

Now patch these local primitives $\xi_\infty|_{\D_p}$ together to obtain $\xi_\infty$ using smooth cutoff functions with gradients supported in 
$\bigsqcup_{p\in\mathfrak p} \mathbb A_p$.
The assertions about the supports of $\xi_\infty$ and $\beta_\infty$ are now obvious. Since $d_{A_\infty} \alpha_\infty = 0 $ and $F_{A_\infty}=0$ we 
obtain $d_{A_{\infty}}\beta_{\infty}=0$. Last, since $d_{A_\infty}^* \alpha_\infty=0$ we see that $R=-d_{A_\infty}^*d_{A_\infty} \xi_\infty$ has support in 
$\bigsqcup_{p\in\mathfrak p}\mathbb{A}_p$. 
\end{proof}

We can view $\beta_\infty$ as an ungauged tangent vector to the space of approximate solutions at $(A^\app_t, \Phi^\app_t)$. Indeed, 
$(A_t^{\app},\Phi_t^{\app}) 
=(A_{\infty},\Phi_{\infty})^{g_t^{\app}}$ for a (singular) complex gauge transformation,  which we can assume equals the identity 
outside each $\D_p(1/2)$. Hence its differential preserves $\beta_{\infty}$. This yields for each $t$ the gauged tangent vector 
\begin{equation}\label{eq:verttangvecgauge}
(\alpha_t,\varphi_t):=(\beta_{\infty},0)-D_t^1\xi_t 
\end{equation}
where $\xi_t\in\Omega^0(\mathfrak{su}(E))$ is the unique solution to 
\begin{equation}\label{eq:vertgaugefix}
\calL_t\xi_t=(D_t^1)^* (\beta_\infty,0) =  d_{A_\infty}^* \, \beta_\infty =E,
\end{equation}
where we have assumed without loss of generality that $A_t^\app=A_\infty$ on $\supp \beta_\infty$. 
To estimate  $\xi_t$, we write $\xi_t=(\xi_\infty+\xi_t) - \xi_\infty$ and consider the equivalent equation 
\begin{equation}
	\calL_t (\xi_t + \xi_\infty) = R_t,
\end{equation}
where
\[
R_t = R + \calL_t\xi_\infty = \calL_t\xi_\infty-\Delta_{A_\infty}\xi_\infty.
\]
However, recall that $\calL_\infty \xi_\infty = \Delta_{A_\infty} \xi_\infty$ since $\xi_\infty$ commutes with $\Phi_\infty$ 
and $d_{A_\infty}^* \xi_\infty = 0$. Thus 
\[
R_t =(\calL_t-\calL_\infty)\xi_\infty = ( \Delta_{A_t} - \Delta_{A_\infty}) \xi_\infty  + t^2 (M_{\Phi_t} - M_{\Phi_\infty}) \xi_\infty.
\]

\begin{proposition} 
Both 
\[
R_t = \sum_{j=1}^\infty \rho_{j,t}(z) t^{1-2j/3} + \calO(e^{-\beta t}) \quad \mbox{and} \quad 
\xi_t + \xi_\infty = \sum_{j=0}^\infty b_{j,t}(z) t^{( - 1 - 2j)/3} + \calO(e^{-\beta t})
\]	
are convergent sums of exponential packets of weights $1-2j/3$ and $( - 1 - 2j)/3$, respectively.
\label{prop:decay_diff}
\end{proposition}
\begin{proof}
As in Lemma \ref{lem:primitive}, near each $p$, 
\[
\xi_\infty  = \sum_{j=0}^\infty \xi_{\infty,j}(z)t^{(-1-2j)/3},
\]
where the coefficients are independent of $t$. Next, following the rescaling calculation in \S 5.1, 
\[
\calL_t -  (\Delta_\infty + t^2 M_\infty) = t^{4/3} ( (\calL_\varrho - \calL_\infty) + (M_\varrho - M_\infty)),
\]
and the coefficients of $\Delta_\infty - \Delta_\varrho$ and $M_\infty - M_{\varrho}$ are weighted exponential packets. 
The conclusions then follow immediately. 
\end{proof}

We next analyze the difference between the initial vertical tangent vector $(\alpha_{\infty}, 0)$ 
and the gauged one, 
\begin{equation}\label{eq:gaugedvertfam}
(\alpha_t,\varphi_t):= (\beta_{\infty},0)-D_t^1\xi_t.  
\end{equation}
%of vertical tangent vectors to $\mathcal M^{\appr}$.

\begin{proposition}
The difference $(\alpha_t, \varphi_t) - (\alpha_\infty,0) $ is a convergent sum of exponential packets,
\[
(\alpha_t, \varphi_t) - (\alpha_\infty,0) = \sum_{j=0}^\infty c_{j,t}(z) t^{(1-2j)/3} + \calO(e^{-\beta t}).
\]	
\end{proposition}

\begin{proof}
By \eqref{eq:verttangvecgauge}, 
\[
(\alpha_t,\varphi_t)=(\beta_{\infty},0)-D_t^1\xi_t = 
(\alpha_\infty,0) - (d_{A_\infty} \xi_\infty, 0) - (d_{A_t}\xi_t, t [\Phi_t, \xi_t]). 
\]
% \begin{align*}
% (\alpha_t,\varphi_t)&=(\beta_{\infty},0)-D_t^1\xi_t\\
% &= (\alpha_\infty,0) + \bigl(d_{A_\infty}\xi_\infty-d_{A_t}\xi_t, t [\Phi_t, \xi_t]\bigr).
% \end{align*}
Now write
\[
d_{A_\infty} \xi_\infty + d_{A_t} \xi_t = (d_{A_\infty} - d_{A_t}) \xi_\infty + d_{A_t} (\xi_\infty+\xi_t);
\]
% \[
% d_{A_\infty}\xi_\infty-d_{A_t}\xi_t = (d_{A_\infty}-d_{A_t})\xi_\infty + d_{A_t}(\xi_t - \xi_\infty)
% \]
and observe also that 
\[
d_{A_\infty} - d_{A_t} = - \bigl(2 f_t(r) - \tfrac{1}{4} \bigr) \begin{pmatrix} i & 0 \\ 0 & -i \end{pmatrix} d\theta.
\]
Since $2 f_t(r) - \tfrac{1}{4}=\eta(\varrho)$ and $|d\theta|=r^{-1}=\varrho^{-2/3} t^{2/3}$, we see that
\[
(d_{A_\infty} - d_{A_t})\xi_\infty = \sum_{j=1}^\infty b_{j,t}(z) t^{(1-2k)/3} + \calO(e^{-\beta t})
\] 
is a sum of exponential packets, and by Proposition \ref{prop:decay_diff}, so is
\[
d_{A_t}(\xi_\infty + \xi_t) = \sum_{j=1}^\infty \tilde{b}_{j,t}(z) t^{(1-2k)/3} + \calO(e^{-\beta t}).
\]
This shows that $d_{A_\infty}\xi_\infty+d_{A_t}\xi_t$ has the correct form. 

For the other term, note that $[\Phi_\infty, \xi_\infty] = 0$, so that
\[
t [\Phi_t, \xi_\infty] = t [ \Phi_t - \Phi_\infty, \xi_\infty],
\]
and since this difference of Higgs fields is a weighted exponential packet, the same conclusion holds.
\end{proof}
\begin{corollary}
\[
\|(\alpha_t, \varphi_t) \|^2_{L^2(X)} = \|  (\alpha_\infty,0)\|^2_{L^2(X)}+\sum_{j =0}^\infty C_j t^{-(2+j)/3} + \calO(e^{-\beta t})
\]
as $t \to \infty$; in particular $\|(\alpha_t, \varphi_t)  \|_{L^2}^2-\|  (\alpha_\infty,0)\|_{L^2}^2 = \calO(t^{-2/3})$. 
\end{corollary}

If $(\alpha_t^{(j)}, \varphi_t^{(j)})$, $j = 1,2$, are two gauged vertical tangent vectors, then
\begin{multline*}
\langle (\alpha_t^{(1)}, \varphi_t^{(1)}), (\alpha_t^{(2)}, \varphi_t^{(2)})\rangle_{L^2}  \\=
\langle (\alpha_t^{(1)}, \varphi_t^{(1)}), (\alpha_t^{(2)}, \varphi_t^{(2)})\rangle_{\sfl} + 
\sum_{j=0}^\infty S_j'' ((\alpha_t^{(1)}, \varphi_t^{(1)}), (\alpha_t^{(2)}, \varphi_t^{(2)}) t^{-(2+j)/3}
\end{multline*}

We make some comments about why these expansions may be differentiated.  Note first that by construction, $\xi_\infty$ is smooth
on $S_q$.  The term $R = d_{A_\infty}^*\beta_\infty= - d_{A_\infty}^*d_{A_\infty} \xi_\infty$ is smooth away from the zero section and
has a polyhomogeneous singularity there. The operator $\calL_t$ varies smoothly with the spectral curve, and the derivatives
of its coefficients with respect to $t\del_t$ do not change form.  As we have seen, this ensures that the
solution $\xi_t$ to $\calL_t \xi_t = R$ also has a smooth expansion.  This allows us to conclude that all the expansions
in this section may be differentiated.

\section{Asymptotics of mixed terms}\label{sect:asympcross}
The horizontal and vertical directions are orthogonal with respect to the semiflat metric, but the $L^2$ metric has some
nontrivial mixed terms.  We now study their asymptotics. 

We have proved above that if $v^{\hor}$ and $w^{\verti}$ are horizontal and vertical tangent vectors in Coulomb gauge
above $t^2 q$, then 
\[
\frac{1}{t} v^{\hor} = \frac{1}{t} v^{\hor}_\infty + \sum_{j=0}^\infty v_j t^{(1-2j)/3} + \calO(e^{-\beta t}),
\]
\[
w^{\verti} = w^{\verti}_\infty + \sum_{j=0}^\infty w_j t^{(1-2j)/3} + \calO(e^{-\beta t}),
\]
where the $v_j$ and $w_j$ are exponential packets.     In analyzing the inner product between a vertical
and a horizontal vector, the only terms potentially of concern are those of the form
\[
\int_X \langle \mu_t(|q|_k) |q|_k^{-1/2} \dot q, \eta \rangle
\]
where $\mu_t$ is an exponential packet and $\eta$ is one of the terms in the expansion of $D_t^1\xi_t$.  To analyze
such an expression, write $\dot q  =\dot f dz^2 = 4 \dot f w^2 dw^2$, $|q|_k^{1/2} = |\lambda|= |w| |dz| = 2 |w|^2 |dw|$
and $\eta = h dz^2 = 4 h w^2 dw^2$. Then the integrand becomes 
\[
\mu_t(|q|_k) |q|_k^{-1/2} \langle \dot q, \eta \rangle =  8 \mu_t(|q|_k) \Re(\dot f \bar h) |w|^2,
\]
which is smooth on $S_q$ and smooth as $q$ varies. 

 In summary, we obtain
\begin{corollary}
\[
\langle t^{-1} v^{\hor}, w^{\verti} \rangle_{L^2} =
\langle t^{-1} v^{\hor}, w^{\verti} \rangle_{\sfl} + 
%\sum_{j=0}^\infty (C_j' t^{-1-2j/3} + C_j'' t^{(-2-2j)/3}) + \calO(e^{-\beta t}).
\sum_{j=0}^\infty C_j t^{-(2+j)/3} + \calO(e^{-\beta t}).
\]
\end{corollary}  
The same types of arguments as before, relying on Proposition \ref{global_exp_packet}, show that these
expansions are convergent and may be differentiated at will. 

\section{Proof of Theorem \ref{expansiontheorem}}\label{subsect:tangentcorrect}
We now come to the final steps in the proof of Theorem \ref{expansiontheorem} by showing 
that the true moduli space $\mc M'$ is an exponentially small perturbation of the approximate moduli space 
$(\mc M^{\app})'$.  More specifically, we construct a diffeomorphism $\calF: \mc M'  \to (\mc M^{\app})'$  such 
that the difference between the pullback of the $L^2$ metric on $(\mc M^{\app})'$ and the $L^2$ metric on 
$ (\mc M^{\app})'$ decays exponentially as $t \to \infty$. 

The subtleties in the discussion below involve gauge choices, so we describe the procedure carefully. 
Recall from \S 4 that we have actually been working at the level of slices in the premoduli spaces.
Thus the construction of the family of approximate solutions corresponds to a diffeomorphism
$\calK_1: \calS_\infty \to \calS^\app$, while the deformation to a true solution corresponds to
a further map $\calK_2: \calS^\app \to \calS$.  The parametrization of a neighborhood in $\calM'$
by a neighborhood in $\calM_\infty'$ is represented by the composition $\calK_2 \circ \calK_1$,
while the diffeomorphism $\calF$ is induced by $\calK_2^{-1}$. 

We must do two things: first we show that $\calK_2$ is indeed smooth, and second, we compute
the induced map on gauged tangent vectors. 

The first of these is a straightforward extension from the original existence theorem. Indeed,
we obtain the complex gauge transformation $\gamma_t$ for which $\exp (\gamma_t)( S_t^\app) = S_t$ 
by writing the first part of the hyperk\"ahler moment map $\mu$ as a nonlinear map acting on $\gamma$ and 
expanding this equation in a Taylor series about $\gamma = 0$. This takes the form
\[
\calL_t \gamma = \mu( S_t^\app) + \calQ(\gamma),
\]
where $\calQ$ is a smooth function of $\gamma$ (but not its derivatives) which vanishes 
quadratically as $\gamma \to 0$.  The map $\calL_t$ is invertible as a map on hermitian infinitesimal
gauge transformations for each $S_t^\app$, and it is a standard matter to show that its inverse depends 
smoothly on $S_t^\app$. Furthermore, the error term $\mu(S_t^\app)$ is $C^0$-bounded  by $C e^{-\beta t}$. The inverse function theorem applies directly to prove that there exists
a smooth map 
 $$ \calT: \mathcal B \to \Omega^0(X,i\mathfrak{su}(E)),\quad  S_t^\app \mapsto \gamma_t = \calT(S_t^\app), $$
%$\calS^\app \ni S_t^\app \mapsto \gamma_t = \calT(S_t^\app)$
defined on the ball $\mathcal B\subset  \calS^\app$  about $0$ of some radius 
$C' e^{-\beta t}$, such that 
\[
\mu( \exp ( \calT(S_t^\app))  \equiv 0.
\]
This proves the first claim. 

The next step, which is slightly more difficult, is to show that if $v_t = (\alpha_t, \varphi_t)$ 
is a tangent vector to the premoduli space of approximate solutions which satisfies the
gauge fixing condition, then there is a well-defined tangent vector $v_t' = (\alpha_t' , \varphi_t')$
to the space of solutions of the Hitchin equation which is also in gauge, and moreover that
\[
\|v_t - v_t'\| \leq C e^{-\beta t}
\]
for some $\beta > 0$. 

This map is a composition of $d\calK_2$ and a further map to put the result into gauge. 
Thus supposing that $v_t$ satisfies the gauge condition, we first note that the estimates for
the field $\gamma_t$ imply that $\|(d\calK_2 - \mathrm{Id}) v_t\| \leq C e^{-\beta t}$.  
The fact that $v_t$ is in gauge with respect to $S_t^\app$ means that $ d\calK_2 (v_t)$ is nearly in gauge with respect to $S_t$, or more specifically, the correction
field $\xi_t$ satisfies $(D_t^1)^* (d\calK_2(v_t) - D_1^t \xi_t) = 0$, i.e., 
\[
\calL_t \xi_t = (D_1^t)^* d\calK_2 (v_t).
\]
Using that the operator norm of $\calL_t^{-1}$ is uniformly bounded in $t$, cf.\ \cite[Proposition 5.2]{msww14},
it then follows that  $\xi_t$ is bounded in norm by $Ce^{-\beta t}$, and hence the gauged image vector  $d\calK_2(v_t) - 
D_t^1 \xi_t$ is within this exponentially small distance from $v_t$.  

The conclusion of the above estimates is that the gauged tangent vectors to $\calM'$ are 
exponentially close to the gauged tangent vectors to $(\calM^\app)'$.

\medskip 

By identifying $\calM'$ via the diffeomorphism $\phi$ in \eqref{isometryphi} with a torus fibration over $(0,\infty)\times \calS'$ we decompose
\[
T\calM'=T^r\calM'\oplus T^h\calM'\oplus T^v\calM'.
\]
Here the radial subspace $T^r\calM'$ is spanned by horizontal lifts of $\partial_t$, $T^h\calM'$ is spanned by horizontal lifts of tangent vectors to $\calS'$, and $T^v\calM'$ is the vertical tangent bundle. Let $T^{r,\ast}\mathcal M$, $T^{h,\ast}\mathcal M$ and $T^{v,\ast}\mathcal M$ denote the respective dual bundles.

The exponentially small correction from $\mathcal M^{app}$ to $\mathcal M$ constructed in this section, combined with the vertical, horizontal, and radial metric estimates of sections \textsection\textsection \ref{sect:asympradial}--\ref{sect:asympcross} imply finally the
%This implies finally the
\begin{theorem}
There is a decomposition
\[
g_{L^2}  - g_{\sfl}  = h_{rr}+t^2 h_{hh}+  h_{vv}+ h_{rv} +t h_{rh} +t h_{hv}
\]
where
\begin{align*}
 h_{rr}\in\Gamma({\bigodot}^2T^{r,\ast}\calM') ,\quad    h_{hh}\in\Gamma({\bigodot}^2T^{h,\ast}\calM'),\quad   h_{vv}\in\Gamma({\bigodot}^2T^{v,\ast}\calM'),\\
   h_{rh}\in\Gamma(T^{r,\ast}\calM'{\bigodot}T^{h,\ast}\calM'),\quad   h_{rv}\in\Gamma(T^{r,\ast}\calM'{\bigodot}T^{v,\ast}\calM'),\\
       h_{hv}\in\Gamma(T^{h,\ast}\calM'{\bigodot}T^{v,\ast}\calM') 
\end{align*}
with convergent expansions 
\begin{equation*}
 \begin{aligned}
h_{rr} &= t^{-\frac{5}{3}}a_{rr}+\calO(e^{-\beta t}),  \qquad   & h_{rh} = \sum_{j=0}^{\infty}t^{-(3+j)/3}a_{rh}^j+\calO(e^{-\beta t}), \\
h_{hh} & = \sum_{j=0}^{\infty} t^{-(2+j)/3}a_{hh}^j+\calO(e^{-\beta t}),
& h_{rv} = \sum_{j=0}^{\infty}t^{-(3+j)/3}a_{rv}^j+ \calO(e^{-\beta t}),\\ 
h_{vv} & = \sum_{j=0}^{\infty}t^{-(2+j)/3}a_{vv}^j+\calO(e^{-\beta t}), 
& h_{hv} = \sum_{j=0}^{\infty} t^{-(2+j)/3}a_{hv}^j+\calO(e^{-\beta t})
\end{aligned}
\end{equation*}
where  $a_{rr}$, $a_{hh}^j$, $a_{vv}^j$, $a_{rh}^j$, $a_{rv}^j$ and $a_{hv}^j$ are
$t$-independent tensors.
\end{theorem}

Inserting the expansions for these various parts one obtains  Theorem \ref{expansiontheorem} as stated in the introduction.

\end{document}